\documentclass[12pt,a4paper]{article}
\usepackage{amssymb,amsmath,amsthm}
\usepackage{latexsym}

\newtheorem{theorem}{Theorem}
\newtheorem{lemma}{Lemma}
\newtheorem{proposition}{Proposition}
\newtheorem{corollary}{Corollary}
\newtheorem{remark}{Remark}
\numberwithin{equation}{section}
\numberwithin{theorem}{section}
\numberwithin{lemma}{section}
\numberwithin{proposition}{section}
\numberwithin{corollary}{section}
\numberwithin{remark}{section}

\begin{document}
\title{Spatial pointwise behavior of time-periodic Navier-Stokes flow 
induced by oscillation of a moving obstacle}
\author{Toshiaki Hishida\thanks{
Partially supported by the Grant-in-aid for Scientific Research 18K03363
from JSPS} \\
Graduate School of Mathematics \\
Nagoya University \\
Nagoya 464-8602, Japan \\
\texttt{hishida@math.nagoya-u.ac.jp} \\
}
\date{}
\maketitle
\begin{abstract}
We study the spatial decay of time-periodic Navier-Stokes flow 
at the rate $|x|^{-1}$ with/without wake structure in 3D exterior 
domains when a rigid body moves periodically in time. In this 
regime the existence of time-periodic solutions was established 
first in the 2006 paper by Galdi and Silvestre, 
however, with 
little information about spatial behavior at infinity 
so that uniqueness of solutions was not available.
This latter issue has been addressed by Galdi, who 
has recently succeeded in construction of a unique time-periodic solution 
with spatial behavior mentioned above if translational and 
angular velocities of the body fulfill, besides smallness and regularity, 
either of the following assumptions: (i) translation or rotation 
is absent; (ii) both velocities are parallel to the same constant 
vector. 
This paper shows the existence of a unique 
time-periodic Navier-Stokes flow in the small with values in the 
weak-$L^3$ space and then deduces the desired pointwise decay 
of the solution under some condition on the rigid motion of the 
body, that covers the cases (i), (ii) mentioned above.
\end{abstract}

\section{Introduction}
\label{intro}

Let $D$ be an exterior domain in $\mathbb R^3$ with smooth boundary
$\partial D\in C^{1,1}$.
An obstacle is identified with a compact set (with nonempty interior)
$\mathbb R^3\setminus D\subset B_1$, 
and it is assumed to be a 
rigid body moving in a viscous incompressible fluid filling the whole $\mathbb R^3$ with 
prescribed time-dependent rigid motion.
Then the fluid motion in a reference frame attached to the moving body is
described by the Navier-Stokes system (see \cite{G02})
\begin{equation}
\begin{split}
\partial_t {\mathcal U}+{\mathcal U}\cdot\nabla {\mathcal U}
&=\Delta {\mathcal U}+(\eta+\omega\times x)\cdot\nabla {\mathcal U}-\omega\times {\mathcal U}-\nabla p, \\
\mbox{div $\mathcal U$}&=0, \\
{\mathcal U}|_{\partial D}&=\eta+\omega\times x, \\
\lim_{|x|\to\infty}{\mathcal U}&= 0 
\end{split}
\label{NS}
\end{equation}
in $D\times \mathbb R$,
where ${\mathcal U}={\mathcal U}(x,t)\in\mathbb R^3$ and $p=p(x,t)\in\mathbb R$ respectively denote the velocity 
and pressure of the fluid, whereas $\eta=\eta(t)\in\mathbb R^3$ and $\omega=\omega(t)\in\mathbb R^3$ 
are translational and
angular velocities of the body in that frame.
Those velocities $\eta$ and $\omega$ are given and already involved in the equation of motion 
by the change of variables
as well as at the boundary $\partial D$ as the no-slip condition, where $\eta+\omega\times x$
is called the rigid motion. 
Our interest in the present paper is focused on the case when both $\eta$ and $\omega$
are periodic.
We then show that such oscillation of the body induces
time-periodic motion of the fluid with the same period as we would expect.
One can also put a periodic external force in the equation of motion, 
however, it is assumed to be absent for simplicity.

Analysis of periodic solutions covers 
the study of steady flows
with constant velocities $\eta,\,\omega$.
Since spatial decay properties at the rate $|x|^{-1}$ of at least small flows 
in the steady-state regime are well known \cite{BM95, Fi65, G03, GS07a, GS07b, NP}, 
such properties of periodic
flows should be of particular interest.
Let us recall briefly those properties of steady flows in 3D, where even asymptotic structure
is known and this interprets the optimality of the decay rate $|x|^{-1}$ 
unless the total net force, which we denote by $N$, exerted by the fluid to the obstacle
is identically zero.
In his celebrated 1933 paper, 
Leray \cite{Le}
showed the existence of steady flows having finite Dirichlet integral without any smallness
of data, however, with less information about asymptotic behavior at infinity so that uniqueness
of solutions is hopeless.
Later on,  
Finn \cite{Fi65} proved the existence of a 
unique steady flow, called physically reasonable solution,
with anisotropic pointwise decay structure when the translational velocity $\eta\neq 0$ is a small constant
and $\omega=0$.
The leading profile of his solution is the Oseen fundamental solution that exhibits the wake region
behind the body, and the coefficient of the profile is given by the net force $N$ mentioned above.
The reason why the leading profile comes from the linear part is better decay of the flow like $|x|^{-2}$ 
outside the wake region, while it decays at the rate $|x|^{-1}$ inside the wake.
When the translation is absent, that is not the case and the scale-critical rate $|x|^{-1}$ leads to
the balance between the linear part and the nonlinearity.
This suggests that the leading term of small solutions having no wake structure
is singled out from the set of steady self-similar Navier-Stokes flows,
which consists of all Landau solutions parametrized by a vector, which we call label, due to insight given by \v Sver\'ak \cite{Sv}.
In fact, the leading term is a Landau solution with label $N$ when the body is at rest (\cite{KSv}),
while it is another Landau solution with label $(N\cdot\frac{\omega}{|\omega|})\frac{\omega}{|\omega|}$
when the body is rotating with constant angular velocity $\omega\neq 0$
but $\eta=0$ (due to \cite{FH11} and then refined by \cite{FGK}).
See also Galdi \cite{G-b} and expository articles \cite{GN-hb, H-hb}
by Galdi, Neustupa and the present author for more details.
Similar asymptotic structure to what is mentioned above still holds even for time-periodic
Navier-Stokes flows, that we are going to discuss in what follows, when the body is at rest or translating
periodically along a constant direction,
see Galdi and Kyed \cite[Section 4.9]{GK-hb}.

There is an extensive literature on time-periodic Navier-Stokes flows, see 
a comprehensive survey \cite{GK-hb} 
and the references therein.
Here, the only articles to be cited are concerned especially with the exterior problem in 3D.
Let us begin to mention the case when the body is at rest, see \cite{MP, Y, GSo, KMT}. 
In his profound paper \cite{Y}, Yamazaki succeeded in construction of a unique periodic solution 
within the class $L^{3,\infty}$ 
by use of sharp temporal decay estimates of the Stokes semigroup, that is in fact a refinement of the
approach developed by Kozono and Nakao \cite{KN}, where 
$L^{3,\infty}$ denotes the Lorentz space
(weak-$L^3$ space).
In \cite{Y} the external force is assumed to be of divergence form $f=\mbox{div $F$}$ with
$F$ being small in $L^{3/2,\infty}$.
When the external force fulfills suitable pointwise decay properties as well as smallness, Galdi and Sohr \cite{GSo}
proved the existence of a unique periodic solution that enjoys a desired decay like $|x|^{-1}$
at infinity.
The same decay property was also deduced by Kang, Miura and Tsai \cite{KMT},
who further found, among others, the asymptotic structure of periodic solutions in which the leading term is still
a Landau solution with label being the time average of the net force $N=N(t)$.

When the body is moving in a time-periodic fashion, the pioneering work is the paper \cite{GS06} by
Galdi and Silvestre, who
proved the existence of weak and strong periodic solutions in $L^2$ by means of the
Galerkin approach together with Leray's invading domains technique,
where the directions
of $\eta(t)$ and $\omega(t)$ are completely general, however, uniqueness of solutions was not available
because of less information about spatial behavior of their solutions at infinity.
Successively, the present author \cite[Corollary 2.1]{Hi09} discussed the existence of a unique periodic motion
of the fluid performed by external forcing along the same way as in \cite{Y} 
combined with estimates from \cite{HShi09} in the restricted case
when $\omega$ is a constant vector
and $\eta =0$.
This result was revisited by \cite{Ng14} (by means of another method) and 
extended by \cite{GHN16} to the case 
when both $\eta$ and $\omega$ are
constant vectors and are parallel to each other.
The relevant case in which at least one of $\eta(t)$ and $\omega(t)$ is in fact periodic rather than constant 
was studied deeply 
by Galdi \cite{G20, G20a, G20b, G-new}
as well as by Eiter and Kyed \cite{EK}.
The latter authors developed the $L^q$-theory for strong solutions, while
the former author made it clear that a periodic solution which decays
at the rate $|x|^{-1}$ 
uniformly in $t$ with/without wake structure exists uniquely if $\eta$ and $\omega$
fulfill, besides smallness and suitable regularity, either of the following assumptions:
(i) $\eta=0$ or $\omega=0$;
(ii) $\eta(t)$ and $\omega(t)$ are parallel to the same constant vector.

The results due to Galdi mentioned above interest us 
and this article
is indeed inspired by his papers.
We are aiming at
the existence of a unique periodic solution to \eqref{NS}
in the small with values in $L^{3,\infty}$ for general $\eta$ and $\omega$ (as in \cite{GS06})
and then to deduce the desired pointwise decay rate $|x|^{-1}$
of the solution uniformly in $t$ under the additional conditions \eqref{wake-cond0}--\eqref{om-para}
below on the rigid motion of the body, that unify the cases (i), (ii) above
found by Galdi, see Remark \ref{compare-galdi}.
Let $\{\Phi(t,s)\}_{t,s\in\mathbb R}$ be the family of evolution matrices to the ordinary differential equation
$\frac{d\phi}{dt}=-\omega(t)\times\phi$.
Note that each $\Phi(t,s)$ is an orthogonal $3\times 3$ matrix 
and that $\Phi(t+l,s+l)=\Phi(t,s)$ for all $t,\,s\in\mathbb R$
as long as 
$\omega(t)$ is $l$-periodic with some $l>0$.
Our conditions now read
\begin{equation}
\sup_{-\infty<s<t<\infty}\left|\int_s^t\big\{\Phi(t,\tau)\eta(\tau)-\zeta\big\}\,d\tau\right|<\infty
\label{wake-cond0}
\end{equation}
with some constant vector $\zeta\in\mathbb R^3$ and 
\begin{equation}
\mbox{$\omega(t)$ being parallel to $\zeta$ for every $t\in\mathbb R$}.
\label{om-para}
\end{equation}
The latter condition is not needed when $\zeta=0$.
It then turns out
under a smallness assumption on
$\eta$ and $\omega$ that 
the periodic solution $u(t)$ obtained above decays like
\begin{equation*}
\sup_{t\in\mathbb R}|u(x,t)|=O(|x|^{-1})
\end{equation*}
as $|x|\to\infty$ and that the solution exhibits a wake structure in the direction $\zeta$,
see \eqref{pointwise}, if $\zeta\neq 0$.  
It should be emphasized that the similar condition to \eqref{wake-cond0} is already hidden
at the level of linear analysis in a sequence of papers \cite{G20, G20a, G20b, G-new} by Galdi.
The novelty of \eqref{wake-cond0} is that the evolution matrices $\Phi(t,s)$ are involved there to
take into account the interaction between the translation and rotation of the body.
Differently from analysis developed by Galdi, who establishes a complete linear theory
with the desired spatial behavior and then proceed to the Navier-Stokes system,
the conditions \eqref{wake-cond0}--\eqref{om-para} are not used until the final stage in which
the nonlinear flow with the desired behavior is reconstructed.

In \cite{Hi18, Hi20} the present author has developed temporal decay estimates
of the evolution operetor $T(t,s): u(s)=f\mapsto u(t)$, 
that provides a unique solution in $L^q$ to the initial value problem for the linearized system
\begin{equation}
\begin{split}
\partial_tu-\Delta u-(\eta+\omega\times x)\cdot\nabla u+\omega\times u+\nabla p&=0, \\
\mbox{div $u$}&=0, \\
u|_{\partial D}&=0, \\
\lim_{|x|\to\infty} u&=0, \\ 
u(\cdot,s)&=f
\end{split}
\label{linearized}
\end{equation}
in $D\times (s,\infty)$.
Estimates obtained there recover completely the corresponding results
for the autonomous case
(Stokes and Oseen semigroups with/without rotating effect)
\cite{I, MSol, KShi, HShi09, Shi08}.
Once we have those, it is obvious to accomplish the former purpose (existence of a unique solution
for general $\eta$ and $\omega$) mentioned in the preceding paragraph
as long as we just follow the approach
developed by Yamazaki \cite{Y}.
In fact, it will be described just in the first half of subsection 3.1 and, as in \cite{Y},
the solutions constructed there cover not only periodic solutions
but the ones which are bounded on the whole time axis $\mathbb R$ such as almost periodic solutions.
Moreover, it is even possible to show the asymptotic stability of small periodic Navier-Stokes flow
with respect to small initial disturbance being in $L^{3,\infty}$ along the similar manner to
\cite{HShi09, GH, TT} as well as \cite{Y} although this question is not addressed here.
Therefore, the main issue of the present paper must be the latter purpose, 
that is, spatial pointwise decay estimates of periodic solutions (and even solutions
that are bounded on the whole time axis).

The strategy is to employ a cut-off technique to 
reduce our study to the whole space problem and then to reconstruct a solution with desired decay property.
Finally, we identify a given 
periodic solution with the solution reconstructed above near infinity
by uniqueness of solutions.
Although this procedure itself is more or less standard, difficulties stem not only from the non-autonomous character
but also from the fact that local regularity of periodic solutions 
constructed by the method of \cite{Y} is very little,
so that the argument must be delicate.
To overcome those diffculties, first of all,
we make efforts to deduce further (however, still modest) regularity
along with weak form involving the associated pressure, 
see \eqref{weak-ext}, with the aid of
regularity theory of the evolution operator
developed in \cite[Section 5]{Hi20}.
At the level of this weak form, we carry out a cut-off procedure to get a similar weak form for
the whole space problem.
Let $U(t,s): u(s)=f\mapsto u(t)$ be the evolution operator, which is the solution operator
to the initial value problem for the linearized system \eqref{linearized} in the whole space
$\mathbb R^3\times (s,\infty)$, where the boundary condition at $\partial D$ is of course removed.
With the weak form above for the whole space problem at hand,
it is necessary to take 
a solution $U(t,s)^*\psi$
to the backward problem for
the adjoint system with final data $\psi$ 
as a test function to obtain an integral equation with use of  
$U(t,s)$, where $U(t,s)^*$ denotes the adjoint evolution operator.
To justify this procedure, a density property plays an important role, see Lemma \ref{density}.
In view of the structure of explicit representation formula of $U(t,s)P_{\mathbb R^3}$ 
with $P_{\mathbb R^3}$ being the Fujita-Kato projection in the whole space,
we are led to the conditions \eqref{wake-cond0}--\eqref{om-para}  
to reconstruct a solution with the desired spatial decay property. 

The paper is organized as follows.
After some preliminaries including some knowledge of the evolution operator,
we present the main theorems in section 2.
The existence, uniqueness and regularity of solutions (Theorem \ref{main1}) together with weak form
are discussed in section 3.
The final section is devoted to deduction of pointwise behavior of the solution (Theorem \ref{main3}).

\section{Results}
\label{result}

\noindent
{\bf 2.1. Notation}
\medskip

We start with introducing notation.
Let us fix the exterior domain $D\subset\mathbb R^3$
with boundary $\partial D\in C^{1,1}$.
By $B_\rho=B_\rho(0)$ we denote the open ball with radius $\rho>0$ centered at
the origin, and assume that 
$\mathbb R^3\setminus D\subset B_1$.
Set $D_R=D\cap B_R$ for $R\geq 1$.
Let $\Omega$ be one of the exterior domain $D$ fixed above, the whole space
$\mathbb R^3$ and a bounded domain ($D_R,\,B_R$,...).
The class $C_0^\infty(\Omega)$ consists of all smooth functions that are
compactly supported in $\Omega$.
For $1\leq q\leq \infty$ and integer $m\geq 0$,
$L^q(\Omega)$ and $W^{m,q}(\Omega)$ are the standard
Lebesgue and $L^q$-Sobolev spaces.
The norm of $L^q(\Omega)$ is denoted by
$\|\cdot\|_{q,\Omega}$ and it is 
abbreviated to $\|\cdot\|_q=\|\cdot\|_{q,D}$ for the exterior domain $\Omega=D$ under consideration.
Given $q\in [1,\infty]$, we denote the H\"older conjugate exponent by
$q^\prime=q/(q-1)\in [1,\infty]$.
The space $W^{m,q}_0(\Omega)$ stands for the completion of 
$C_0^\infty(\Omega)$ in $W^{m,q}(\Omega)$, and $W^{-1,q}(\Omega)$
denotes the dual space of $W^{1,q^\prime}_0(\Omega)$,
where $1<q<\infty$. 
The Lorentz space 
$L^{q,r}(\Omega)$ is defined by use of the average function of rearrangement,
see Bergh and L\"ofstr\"om \cite{BL}.
We just mention a characterization of the Lorentz space in terms of real interpolation:
\[
L^{q,r}(\Omega)=\big(L^{q_0}(\Omega), L^{q_1}(\Omega)\big)_{\theta,r}
\]
where $(\cdot,\cdot)_{\theta,r}$ stands for the real interpolation
functor and
\begin{equation}
1<q_0<q<q_1<\infty, \qquad
\frac{1}{q}=\frac{1-\theta}{q_0}+\frac{\theta}{q_1}, \qquad
1\leq r\leq \infty.
\label{interpo}
\end{equation}
Note that the space above generated by the interpolation
is independent of choice of $q_0$ and $q_1$ and that the duality relation
$L^{q,r}(\Omega)^*=L^{q^\prime,r^\prime}(\Omega)$
holds unless $r=\infty$.
The Lorentz space $L^{q,r}(\Omega)$ with \eqref{interpo}
is a Banach space whose norm is denoted by $\|\cdot\|_{q,r,\Omega}$.
It is abbreviated to 
$\|\cdot\|_{q,r}=\|\cdot\|_{q,r,D}$ 
for the exterior domain $\Omega=D$. 
We denote by $\langle\cdot,\cdot\rangle$ various duality pairings over several domains,
which are understood in each context.

Let $X$ be a reflexive Banach space 
and $\langle\cdot,\cdot\rangle$ a duality pairing
between $X$ and $X^*$.
Given interval $I\subset\mathbb R$, we set
$C_{w}(I;\,X)=\{u: I\to X;\; \langle u(\cdot),\psi\rangle\in C(I)\;\forall\psi\in X^*\}$.
The space
$C^1_{w}(I;\,X)$ consists of all
$u\in C_{w}(I;\,X)$ such that 
$\langle u(\cdot),\psi\rangle\in C^1(I)$ for all $\psi\in X^*$
with the following property:
For every $t\in I$ there is a constant $c(t)>0$ satisfying
$\left|\frac{d}{dt}\langle u(t),\psi\rangle\right|\leq c(t)\|\psi\|_{X^*}$.
We then find a function $w\in C_{w}(I;\,X)$ such that
$\frac{d}{dt}\langle u(t),\psi\rangle=\langle w(t),\psi\rangle$
for all $t\in I$ and $\psi\in X^*$.
We write $w(t)=\partial_tu(t)$ in $X$.
If $X$ is not reflexive, then the space $C_w(I;\, X)$
should be replaced by $C_{w^*}(I;\, X)$,
where $\langle\cdot,\cdot\rangle$ denotes the pairing between
$X$ and its pre-dual.
Set $BC(I;\,X)=L^\infty(I;\,X)\cap C(I;\,X)$.
It is then clear to realize the definition of 
the other function spaces,
which will appear in Theorem \ref{main1} and so on.

We use the same symbol for denoting scalar, vector and tensor function spaces
if there is no confusion.
Several constants are denoted by $C$, which may change from line to line.
\medskip

\noindent
{\bf 2.2. Evolution operator}
\medskip

Let $\Omega$ be either the exterior domain $D$ under consideration or
the whole space $\mathbb R^3$.
The class $C_{0,\sigma}^\infty(\Omega)$ consists of all solenoidal 
vector fields being in $C_0^\infty(\Omega)$.
For $1<q<\infty$ we denote by $L^q_\sigma(\Omega)$ the completion of
$C_{0,\sigma}^\infty(\Omega)$ in $L^q(\Omega)$.
When $\Omega=D$, it is characterized as
\[
L^q_\sigma(D)=\{u\in L^q(D);\,\mbox{div $u$}=0,\,\nu\cdot u|_{\partial D}=0\},
\]
where $\nu$ stands for outer unit normal to the boundary $\partial D$
and $\nu\cdot u|_{\partial D}$ is understood in the sense of normal trace.
We also have the Helmholtz decomposition
\[
L^q(\Omega)=L^q_\sigma(\Omega)\oplus
\{\nabla p\in L^q(\Omega);\; p\in L^q_{\rm loc}(\overline{\Omega})\}
\]
see Miyakawa \cite{Miy82}, Simader and Sohr \cite{SiSo} for $\Omega=D$.
By $P_\Omega$ we denote the projection, called the Fujita-Kato projection,
associated with the decomposition above.
We abbreviate $P=P_D$.
When $\Omega =\mathbb R^3$, the Fujita-Kato projection is described as
$P_{\mathbb R^3}=I+{\mathcal R}\otimes {\mathcal R}$ with $I$ being the identity map and 
${\mathcal R}=\nabla (-\Delta)^{-1/2}$
being the Riesz transform.

Let $1\leq r\leq \infty$.
Then, by real interpolation, $P_\Omega$ extends a bounded operator acting
on the Lorentz space $L^{q,r}(\Omega)$.
Following Borchers and Miyakawa \cite[Section 5]{BM95}, we define 
the solenoidal Lorentz space by 
$L^{q,r}_\sigma(\Omega):=P_\Omega\big(L^{q,r}(\Omega)\big)$, that is, the range
of such an extension of the Fujita-Kato projection $P_\Omega$.
Note that the duality relation 
$L^{q,r}_\sigma(\Omega)^*=L^{q^\prime,r^\prime}_\sigma(\Omega)$
still holds for $1\leq r<\infty$
(\cite[Theorem 5.2]{BM95}),
that $C_{0,\sigma}^\infty(\Omega)$ is dense in $L^{q,r}_\sigma(\Omega)$
unless $r=\infty$ (\cite[Theorem 5.5]{BM95}), and that
\[
L^{q,r}_\sigma(\Omega)
=\big(L^{q_0}_\sigma(\Omega), L^{q_1}_\sigma(\Omega)\big)_{\theta,r}
\]
with the same exponents 
as in \eqref{interpo} 
(\cite[Theorem 5.4]{BM95}).

Let us introduce the linearized operator 
associated with \eqref{linearized} in exterior domains.
Let $1<q<\infty$. The Stokes operator $A$ on the space
$L^q_\sigma(D)$ is defined by
\begin{equation}
D_q(A)=L^q_\sigma(D)\cap W^{1,q}_0(D)\cap W^{2,q}(D), \qquad
Au=-P\Delta u.  
\label{stokes}
\end{equation}
It is well known 
that fractional powers $A^\alpha$ are well-defined as closed operators on $L^q_\sigma(D)$.
We refer to \cite{BM90} for detailed studies of them.
The only case used in this paper is $\alpha=1/2$ (square root),
see \eqref{stokes-evo} in the next section.

Suppose
\begin{equation}
\eta,\;\omega\in W^{1,\infty}(\mathbb R;\,\mathbb R^3)
\cap C^1(\mathbb R;\,\mathbb R^3)
\label{eta-om0}
\end{equation}
throughout the paper.
The magnitude of those velocities of the rigid body is denoted by
$\|(\eta,\omega)\|_{W^{1,\infty}}$.
Let $1<q<\infty$.
For each $t\in\mathbb R$, we define
the operator $L_\pm(t)$ 
by
\begin{equation}
\begin{split}
&D_q(L_\pm(t))=\{u\in D_q(A);\; 
(\omega(t)\times x)\cdot\nabla u\in L^q(D)\}, \\
&L_\pm(t)u=-P[\Delta u\pm (\eta(t)+\omega(t)\times x)\cdot\nabla u\mp\omega(t)\times u].
\end{split}
\label{L}
\end{equation}
Then the initial value problem for the
linearized system \eqref{linearized} is written as
\begin{equation}
\frac{du}{dt}+L_+(t)u=0, \quad t\in (s,\infty); \qquad
u(s)=f
\label{linear0}
\end{equation}
in $L^q_\sigma(D)$.
It follows from \cite[(2.12)]{Hi18} that
\begin{equation}
\langle L_\pm(t)u, v\rangle
=\langle u, L_\mp(t)v\rangle
\label{L-adj}
\end{equation}
for all $u\in D_q(L_\pm(t))$ and $v\in D_{q^\prime}(L_\mp(t))$.
Hansel and Rhandi \cite{HR} proved that the operator family
$\{L_+(t);\,t\in\mathbb R\}$ generates an evolution operator
$\{T(t,s);\,-\infty<s\leq t<\infty\}$ on $L^q_\sigma(D)$.
Although they discussed the case 
$t\geq s\geq 0$,
the argument works for $-\infty<s\leq t<\infty$ without any 
change.
See \cite[Proposition 2.1, Section 5]{Hi20} on several regularity properties of the evolution
operator $T(t,s)$, 
which are further developments of studies in \cite{HR}.
The adjoint evolution operator $T(t,s)^*$ provides a solution
to the backward problem for the adjoint system subject to the final condition at $t\in \mathbb R$, that is,
\begin{equation}
-\frac{dv}{ds}+L_-(s)v=0, \quad s\in (-\infty,t); \qquad
v(t)=g
\label{back}
\end{equation}
in $L^q_\sigma(D)$ as long as the final data $g$ are fine, see \cite[Section 2]{Hi20}.

In \cite[Theorem 2.1]{Hi18} and \cite[Theorem 2.1, Theorem 2.2]{Hi20},
under weaker assumptions on the rigid motion than \eqref{eta-om0},
the present author established the following 
decay estimates of the evolution operators $T(t,s)$ and $T(t,s)^*$
for all $t,\,s \in\mathbb R$
with $t>s$ (all $t\in\mathbb R$ concerning \eqref{evo4} below),
$f\in L^q_\sigma(D)$, $g\in L^{q,1}_\sigma(D)$ and 
$F\in L^q(D)^{3\times 3}$: 
\begin{equation}
\|T(t,s)f\|_r\leq C(t-s)^{-(3/q-3/r)/2}\|f\|_q, \qquad
1<q\leq r\leq \infty,\; q\neq \infty
\label{evo1}
\end{equation}
\begin{equation}
\|\nabla T(t,s)f\|_r\leq C(t-s)^{-1/2-(3/q-3/r)/2}\|f\|_q, \qquad
1<q\leq r\leq 3
\label{evo2}
\end{equation}
\begin{equation}
\|\nabla T(t,s)^*g\|_{r,1}
\leq C(t-s)^{-1/2-(3/q-3/r)/2}\|g\|_{q,1}, \qquad
1<q\leq r\leq 3
\label{evo3}
\end{equation}
\begin{equation}
\int_{-\infty}^t\|\nabla T(t,s)^*g\|_{r,1}\,ds
\leq C\|g\|_{q,1}, \qquad
1<q< r\leq 3,\;\frac{1}{q}-\frac{1}{r}=\frac{1}{3}
\label{evo4}
\end{equation}
\begin{equation} 
\|T(t,s)P\,\mbox{div $F$}\|_r\leq C(t-s)^{-1/2-(3/q-3/r)/2}\|F\|_q,
\quad  \frac{3}{2}\leq q\leq r\leq\infty,\; q\neq\infty.
\label{evo5}
\end{equation}
This together with Remark \ref{rem-evo} (i) below
is a set of estimates for later
use although all of them hold for both $T(t,s)$ and $T(t,s)^*$.
Let us summarize some comments on those estimates in the following remark.
\begin{remark}
(i)
One needs the restriction $1<q\leq r\leq 3$ for decay rate \eqref{evo2} as $(t-s)\to\infty$
\cite[Remark 2.1]{Hi20}, while
smoothing rate in \eqref{evo2} as $(t-s)\to 0$ holds true even 
for $1<q\leq r<\infty$;
to be precise, \eqref{evo2} with $C=C(\tau_*)$ for $0\leq s< t\leq \tau_*$
was proved in \cite{HR} and then it was extended by the present author \cite[Proposition 3.1]{Hi18}
to \eqref{evo2} with $C=C(\tau_*)$ for $t-s\leq \tau_*$, where $\tau_*\in (0,\infty)$ is arbitrary.
When $1<q\leq r\in (3,\infty)$, less decay rate $(t-s)^{-3/2q}$ of $\nabla T(t,s)$
for $t-s>1$ is deduced \cite[Theorem 2.1]{Hi20} instead of \eqref{evo2},
but it is not needed in this paper.
The similar comments are valid also for \eqref{evo3} and \eqref{evo5}.

(ii)
As in \cite{HR}, the case $t> s\geq 0$ was studied in \cite{Hi18, Hi20},
but all the arguments work for $-\infty<s<t<\infty$ as well.
Accordingly, estimate given in \cite[Theorem 2.2, (2.26)]{Hi20} for the integral $\int_0^t$ 
can be replaced by
\eqref{evo4}.
The idea of deducing \eqref{evo4} from \eqref{evo3} by use of
real interpolation is due to \cite{Y}.

(iii)
Estimate \eqref{evo5} follows from duality of \eqref{evo2} with
$T(t,s)$ replaced by $T(t,s)^*$ and the semigroup property, see \cite[Proposition 3.3]{GH}.
This tells us that the composite operator $T(t,s)P\,\mbox{\rm div}$
extends a bounded operator from $L^{q}(D)^{3\times 3}$
to $L^r_\sigma(D)$, $r\in (q,\infty)$, and to $L^\infty(D)$.

(iv)
Under the assumption \eqref{eta-om0}, all the constants $C>0$ in \eqref{evo1}--\eqref{evo5}
depend only on 
$q,\, r,\, D$ and $m>0$, where
$\|(\eta,\omega)\|_{W^{1,\infty}}\leq m$ which can be large;
that is, those estimates hold uniformly in such $\eta,\,\omega$.
In what follows, we assume $\|(\eta,\omega)\|_{W^{1,\infty}}\leq 1$ for simplicity to fix the constants
in \eqref{evo1}--\eqref{evo5}.
\label{rem-evo}
\end{remark}

It is readily seen that
\begin{equation}
T(t+l,s+l)=T(t,s) 
\label{per-evo-op} 
\end{equation}
for all $t,\,s\in\mathbb R$ with $t>s$ 
provided that $\eta$ and $\omega$
are $l$-periodic with some period $l>0$.
This is verified from the fact that
$v(t)=T(t+l,s+l)f-T(t,s)f$ 
enjoys
\[
\frac{dv}{dt}+L_+(t)v=0, \quad t\in (s,\infty); \qquad v(s)=0
\]
for all $f\in C_{0,\sigma}^\infty(D)$ that is dense in $L^q_\sigma(D)$.
\medskip

\noindent
{\bf 2.3. Main results}
\medskip

Having $\mathbb R^3\setminus D\subset B_1$ in mind,
we fix a cut-off function $\varphi$ such that
\begin{equation}
\varphi\in C_0^\infty(B_2), \qquad
\varphi(x)=1 \quad(x\in B_1).
\label{cut}
\end{equation}
Set
\begin{equation}
b(x,t)=\frac{1}{2}\,\mbox{rot}\left[\varphi(x)
\big(\eta(t)\times x-|x|^2\omega(t)\big)\right],
\label{lift}
\end{equation}
that is a lift of the rigid motion at the boundary $\partial D$
and satisfies
\begin{equation}
\begin{split}
&\mbox{div $b$}=0, \qquad
b|_{\partial D}=\eta+\omega\times x, \qquad
b(t)\in C_0^\infty(B_2) \\
&b\in BC(\mathbb R;\, W^{2,q}(\mathbb R^3)), \qquad
\partial_tb\in BC(\mathbb R;\, L^q(\mathbb R^3)) \\
&\sup_{t\in\mathbb R}\Big(\|b(t)\|_{W^{2,q}(\mathbb R^3)}
+\|\partial_tb(t)\|_{q,\mathbb R^3}\Big)
\leq C\|(\eta,\omega)\|_{W^{1,\infty}} 
\end{split}
\label{lift-est}
\end{equation}
for every $q\in [1,\infty]$ with some constant $C=C(q)>0$.
Let us look for a solution to \eqref{NS} of the form
\[
{\mathcal U}(x,t)=b(x,t)+u(x,t),
\]
then $u(t)$ obeys
\begin{equation}
\frac{du}{dt}+L_+(t)u+P(u\cdot\nabla u+b\cdot\nabla u+u\cdot\nabla b)=Pf\qquad (t\in \mathbb R)
\label{evo-eq}
\end{equation}
in $L^q_\sigma(D)$, where
\begin{equation}
\begin{split}
f&=\mbox{div $F$} 
=\Delta b+(\eta+\omega\times x)\cdot\nabla b-\omega\times b  
-\partial_tb-b\cdot\nabla b, \\  
F&=\nabla b+b\otimes (\eta+\omega\times x) 
-(\omega\times x)\otimes b-F_0-b\otimes b   
\end{split}
\label{force0}
\end{equation}
with
$F_0(t)=-\nabla (4\pi|x|)^{-1}*\partial_tb(t)\in L^{r,\infty}(\mathbb R^3)$,
$3/2\leq r\leq\infty$ ($L^{\infty,\infty}=L^\infty$).
We then see from \eqref{lift-est} that
\begin{equation} 
\begin{split}  
&f(t)\in C_0^\infty(B_2), \quad 
f\in BC(\mathbb R;\, L^q(\mathbb R^3)), \quad 
F\in BC(\mathbb R;\,L^{r,\infty}(\mathbb R^3))  \\ 
&\|f(t)\|_{q,\mathbb R^3}\leq C\|(\eta,\omega)\|_{W^{1,\infty}}, \qquad
\|F(t)\|_{r,\infty,\mathbb R^3}
\leq 
C^\prime\|(\eta,\omega)\|_{W^{1,\infty}}
\end{split} 
\label{f-F} 
\end{equation}
for every $q\in [1,\infty]$ and $r\in [3/2,\infty]$ with some
constants $C=C(q),\, C^\prime=C^\prime(r)>0$,
where we have assumed 
$\|(\eta,\omega)\|_{W^{1,\infty}}\leq 1$ 
at the beginning 
as already mentioned in Remark \ref{rem-evo} (iv).
Indeed, we will use both forms of the external force, $f$ and $\mbox{div $F$}$.

In view of \eqref{evo-eq}, 
we use fundamental properties of the evolution operator to consider formally
$\partial_s\{T(t,s)u(s)\}$ in $L^q_\sigma(D)$ and then integrate it over the interval $(\tau,t)$.
If one uses \eqref{evo1} together with
$\limsup_{\tau\to -\infty}\|u(\tau)\|_{q_0}<\infty$
for some $q_0\in (1,q)$, see the latter half of the proof of Proposition \ref{q-bdd}, 
the equation that $u(t)$ fulfills is reduced to 
\begin{equation}
u(t)
=\int_{-\infty}^tT(t,s)Pf(s)\, ds -\int_{-\infty}^t
T(t,s)P\,\mbox{div}(u\otimes u+u\otimes b+b\otimes u)(s)\, ds.
\label{NS-int}
\end{equation}

Before stating our main results, we should discuss the condition which ensures that the second integral of
the right-hand side of \eqref{NS-int} makes sense as the Bochner integral.
Note that the
condition $r_0>3$ in \eqref{ensure-bochner} below is reasonable because we have the class of steady flows
in mind (if translation of the body is absent).
\begin{lemma}
Suppose that there are $r_0\in (3,6)$ and $r_1\in (r_0,\infty)$ 
such that
\begin{equation}
u\in L^\infty(\mathbb R;\,L^{r_0}_\sigma(D)\cap L^{r_1}_\sigma(D)).
\label{ensure-bochner}
\end{equation}
Then 
the second integral of the right-hand side of \eqref{NS-int} is Bochner summable in $L^q(D)$ for every 
$q\leq\infty$ satisfying
$2/r_1-1/3<1/q<2/r_0-1/3$.
\label{b-s}
\end{lemma}
\begin{proof}
It suffices to discuss the term $u\otimes u$ because $b$ is quite fine
as in \eqref{lift-est}.
Given $q$ satisfying the condition, one can take $r_2\in (r_0,r_1]$ such that
$2/r_2-1/3<1/q\leq 2/r_2$.
Using $(u \otimes u)(t)\in L^{r_0/2}(D)\cap L^{r_2/2}(D)$ with $r_2/2>r_0/2>3/2$,
we can apply  
\eqref{evo5} for the composite operator $T(t,s)P\,\mbox{div}$
to see the conclusion 
by splitting the integral as
$\int_{-\infty}^{t-1}+\int_{t-1}^t$.
\end{proof}

If $u\in BC(\mathbb R;\,L^q_\sigma(D))$ satisfies
\eqref{NS-int} in $L^q_\sigma(D)$ for some $q\in (3,\infty)$ in which the integral 
makes sense as the Bochner integral,
then $u(t)$ 
is said to be a solution to \eqref{NS-int}.
The following theorem provides a unique solution to \eqref{NS-int},
where $\eta$ and $\omega$ are not necessarily periodic.
When they are periodic, the solution in the following theorem is in fact periodic with the same period,
see Corollary \ref{per-sol}. 
\begin{theorem}
There exists a constant $\delta\in (0,1]$  
such that the following statement holds true:
If $\eta$ and $\omega$ satisfy \eqref{eta-om0} with
$\|(\eta,\omega)\|_{W^{1,\infty}}\leq\delta$, then problem
\eqref{NS-int} admits a unique solution $u(t)$ that enjoys
\begin{equation}
u\in BC_{w^*}(\mathbb R;\,L^{3,\infty}_\sigma(D)), \qquad
\sup_{t\in\mathbb R}\|u(t)\|_{3,\infty}\leq C\|(\eta,\omega)\|_{W^{1,\infty}}
\label{sol-cl}
\end{equation}
with some constant $C=C(D)>0$ as well as
\begin{equation}
u\in BC(\mathbb R;\,L^q_\sigma(D)) 
\cap BC_w(\mathbb R;\, W^{1,q}_0(D))  
\label{reg1}
\end{equation}
\begin{equation}
\sup_{t\in\mathbb R}\big(\|u(t)\|_q 
+\|\nabla u(t)\|_q\big)
\leq C\|(\eta,\omega)\|_{W^{1,\infty}}
\label{reg3}
\end{equation}
for all $q\in (3,\infty)$ with some constant 
$C=C(q,D)>0$.  
Given arbitrary $R>2$, we have
\begin{equation}
u\in C^1_{w}(\mathbb R;\, W^{-1,r}(D_R)), \qquad 
p\in BC_w(\mathbb R;\, L^r(D_R)) 
\label{reg2}
\end{equation}
\begin{equation} 
\sup_{t\in\mathbb R}
\big(\|\partial_tu(t)\|_{W^{-1,r}(D_R)}+\|p(t)\|_{r,D_R}\big) 
\leq C\|(\eta,\omega)\|_{W^{1,\infty}} 
\label{reg4} 
\end{equation}
for all $r\in (1,\infty)$ with some constant $C=C(r,R,D)>0$, 
where the pressure $p(t)$ associated with $u(t)$ is singled out in such a way that
$\int_{D_R}p(t)\,dx=0$.
\label{main1}
\end{theorem}
\begin{remark}
(i)
Uniqueness of solutions
is established within \eqref{sol-cl} or merely within
$L^\infty(\mathbb R;\, L^{3,\infty}_\sigma(D))$
with small norm 
even if the sense of solutions is weaker than
described just before Theorem \ref{main1} and the same as in \cite{Y},
see \eqref{weak-int}.

(ii)
By \eqref{reg1}, especially, boundedness in time with values in $L^q_\sigma(D)$ for every $q\in (3,\infty)$,
the solution obtained in Theorem \ref{main1} satisfies \eqref{NS-int}
in $L^q_\sigma(D)$ for such $q$, 
see Lemma \ref{b-s} and Proposition \ref{q-bdd}.
Note, however, that the Bochner summability of the second integral
of the right-hand side of \eqref{NS-int}
in $L^{3,\infty}_\sigma(D)$ is not available unlike the
case of initial value problem \cite[Definition 2.1]{GH}.
\label{sol-bdd}
\end{remark}
\begin{corollary}
If in particular $\eta$ and $\omega$ are $l$-periodic with some period $l>0$ 
in addition to the assumption of Theorem \ref{main1},
then the solution $u(t)$ obtained in Theorem \ref{main1}
is periodic with the same period. 
The associated pressure $p(t)$ singled out as in Theorem \ref{main1}
for given $R> 2$ is also $l$-periodic.  
\label{per-sol}
\end{corollary}

The following theorem answers to the main issue of this paper.
The condition \eqref{wake-cond} below 
tells us the interaction between the translation and rotation, which enables us to deduce
the pointwise behavior of the solution at infinity.
\begin{theorem}
In addition to \eqref{eta-om0}, suppose there is
a constant vector $\zeta\in\mathbb R^3$ such that
\begin{equation}
M:=
\sup_{-\infty<s<t<\infty}\left|\int_s^t \big\{\Phi(t,\tau)\eta(\tau)-\zeta\big\}\,d\tau\right|<\infty
\label{wake-cond}
\end{equation}
and that 
$\omega(t)$ is parallel to $\zeta$ for every $t\in\mathbb R$
(this latter condition is redundant if $\zeta=0$), where
$\{\Phi(t,s)\}_{t,s\in\mathbb R}$ denotes the family of evolution matrices 
to the system $\frac{d\phi}{dt}=-\omega(t)\times\phi$.
Then there exists a constant $\delta_*\in (0,\delta]$,
with $\delta$ being the constant in Theorem \ref{main1},
such that as long as  
$\|(\eta,\omega)\|_{W^{1,\infty}}\leq\delta_*$,
the solution $u(t)$ obtained in 
Theorem \ref{main1} enjoys the pointwide decay property
\begin{equation}
\sup_{t\in\mathbb R}|u(x,t)|
=O\big(|x|^{-1}(1+|\zeta||x|+\zeta\cdot x)^{-1}\big) 
\label{pointwise}
\end{equation}
as $|x|\to\infty$.
\label{main3}
\end{theorem}
\begin{remark}
In \cite{G20, G20a, G20b, G-new} Galdi constructed a unique periodic Navier-Stokes flow
possessing the pointwise decay property \eqref{pointwise} with
$\zeta=\frac{1}{l}\int_0^l\eta(t)\,dt$
if $\eta$ and $\omega$ fulfill, besides smallness as well as suitable regularity,
either of the following assumptions:
(i) $\eta=0$ or $\omega=0$;
(ii) they are parallel to the same constant vector.
Note that the case (ii) implies that $\Phi(t,\tau)\eta(\tau)=\eta(\tau)$ and, therefore,
the condition \eqref{wake-cond} is met with $\zeta$ specified above.
In addition to the cases he has discovered,
Theorem \ref{main3} covers the following situation:
$\int_0^l\omega(t)\,dt=0$ and the direction of $\omega(t)$ 
is parallel to $\zeta$ for every $t\in\mathbb R$.
In fact, this circumstance leads to 
$\Phi(t+l,0)=\Phi(t,0)$ (which is not true in general) and
\eqref{wake-cond} is met with
$\zeta=\frac{1}{l}\int_0^l \Phi(t,0)^\top\eta(t)\,dt$, where $(\cdot)^\top$ denotes the transposed matrix,
because
\[
\left|\int_s^t\big\{\Phi(t,\tau)\eta(\tau)-\zeta\big\}\,d\tau\right|=
\left|\int_s^t\left\{\Phi(\tau,0)^\top\eta(\tau)-\zeta\right\}\,d\tau\right|
\]
is uniformly bounded in $t,\,s$ for such $\zeta$.
Here, the direction of $\eta(t)$ can be different from $\zeta$.
\label{compare-galdi}
\end{remark}

\section{Proof of Theorem \ref{main1}}
\label{pr1}

\noindent
{\bf 3.1. Weak form of the integral equation}
\medskip

Uniqueness and existence of solutions within $L^\infty(\mathbb R;\,L^{3,\infty}_\sigma(D))$
are established for the weak form
\begin{equation}
\langle u(t),\psi\rangle=
\int_{-\infty}^t\langle (u\otimes u+u\otimes b+b\otimes u-F)(s),\,
\nabla T(t,s)^*\psi\rangle\,ds \;\;
\mbox{for all $\psi\in C_{0,\sigma}^\infty(D)$}
\label{weak-int} 
\end{equation}
as in Yamazaki \cite{Y}.
The argument is exactly the same except for the continuity of the solution in $t$.
In fact, let both $u$ and $\widetilde u$ satisfy \eqref{weak-int} with \eqref{sol-cl},
where weak* continuity in $t$ is not needed for uniqueness, then we use \eqref{evo4} 
($q=3/2,\,r=3$) and $L^{3,\infty}$-estimate of $b$, that follows from
\eqref{lift-est},
together with the duality relation $L^{3,\infty}_\sigma(D)=L^{3/2,1}_\sigma(D)^*$ to obtain
\[
\sup_{t\in\mathbb R}\|u(t)-\widetilde u(t)\|_{3,\infty}
\leq C\|(\eta,\omega)\|_{W^{1,\infty}}\sup_{t\in\mathbb R}\|u(t)-\widetilde u(t)\|_{3,\infty}
\]
which implies $u=\widetilde u$ under the smallness of $\eta$ and $\omega$.
Define the map $\Phi$ by the relation
\[
\langle(\Phi u)(t),\psi\rangle:=\mbox{the right hand-side of \eqref{weak-int}},
\]
then \eqref{evo4}, \eqref{lift-est} and \eqref{f-F} imply the existence of a closed ball 
in $L^\infty(\mathbb R;\,L^{3,\infty}_\sigma(D))$,
whose radius is controlled by $\|(\eta,\omega)\|_{W^{1,\infty}}$ and which is invariant under the action
of the map $\Phi$, provided that $\eta$ and $\omega$ are small enough.
It follows from the same manner as above for uniqueness that $\Phi$ is contractive on the ball.
The fixed point of $\Phi$ provides a solution to the weak form \eqref{weak-int}
with \eqref{sol-cl}, 
however, not to \eqref{NS-int} yet.

Unlike \cite{Y}, continuity \eqref{sol-cl} is proved
only in weak* sense 
because of the non-autonomous character.
See Takahashi \cite[Lemma 4.1, (4.4)]{TT}, in which
$\int_{-\infty}^t$ is replaced by $\int_0^t$ in \eqref{weak-int} since initial value problem is discussed there.
On account of \eqref{evo4} even for the integral $\int_{-\infty}^t$, his argument works well without any change.

We will discuss the additional regularity properties \eqref{reg1}--\eqref{reg4} in the following two subsections.
To this end, we have to show first of all that
$u\in L^\infty(\mathbb R;\,L^q_\sigma(D))$ with some $q>3$, which is
still based on \eqref{weak-int}.
Although this can be deduced for every $q\in (3,\infty)$ by using solely \eqref{weak-int},
what is unpleasant is that
the smallness of $(\eta,\omega)$ depends on $q\in (3,\infty)$ (as in \cite{Hi09}).
Since the only thing we need in this step to proceed to the next step is the property above for some $q>3$
(in fact, we deduce \eqref{reg1} 
for every $q\in (3,\infty)$ in subsection 3.2 without additional smallness),
let us 
show that
\begin{equation}
u\in BC_{w^*}(\mathbb R;\,L^{6,\infty}_\sigma(D)), \qquad
\sup_{t\in\mathbb R}\|u(t)\|_{6,\infty}\leq C\|(\eta,\omega)\|_{W^{1,\infty}}
\label{bdd-zuerst0}
\end{equation}
which combined with \eqref{sol-cl} implies
\begin{equation}
u\in BC_w(\mathbb R;\,L^q_\sigma(D)), \quad
\sup_{t\in\mathbb R}\|u(t)\|_q\leq C\|(\eta,\omega)\|_{W^{1,\infty}}\quad
\mbox{for all $q\in (3,6)$}
\label{bdd-zuerst}
\end{equation}
with some constant $C=C(q,D)>0$.
The idea is the same as in the proof of \cite[Theorem 2.1 (2)]{Hi09} by the present author,
in which the similar property is deduced in the regime of 
rotating obstacle with constant angular velocity as well as external forcing.
We give just a sketch.
Let $u$ be the solution to \eqref{weak-int} obtained above, and let us define an
auxiliary map $Q_u$ by the relation
\[
\langle Q_u[V](t),\psi\rangle=
\int_{-\infty}^t\langle (u\otimes V+u\otimes b+b\otimes u-F)(s),\,
\nabla T(t,s)^*\psi\rangle\, ds \;\;
\mbox{for all $\psi\in C_{0,\sigma}^\infty(D)$}.
\]
Then 
\eqref{evo4} (with $q=3/2,\,6/5$) gives us a fixed
point $V\in L^\infty(\mathbb R;\,L^{3,\infty}_\sigma(D)\cap L^{6,\infty}_\sigma(D))$
of the map $Q_u$ provided 
$\|u(t)\|_{3,\infty}$ is small enough, which is accomplished through \eqref{sol-cl} if
$\eta$ and $\omega$ are still smaller,
where $L^{q,\infty}$-estimate of $b$ (with $q=3,\,6$), that follows from
\eqref{lift-est}, 
and \eqref{f-F} (with $r=3/2,\,2$) are employed.
Since it is unique only within $L^\infty(\mathbb R;\,L^{3,\infty}_\sigma(D))$
and since $u$ itself fulfills $Q_u[u]=u$, we see that $u$ must coincide with $V$,
which implies \eqref{bdd-zuerst0}--\eqref{bdd-zuerst},
where weak* (or weak) continuity also follows from the continuity
\eqref{sol-cl} and the obtained estimates.
The small constant $\delta\in (0,1]$ in Theorem \ref{main1} is determined at this stage.
\medskip

\noindent
{\bf 3.2. Regularity in $x$}
\medskip

We fix $q\in (3,6)$ arbitrarily.
With $u(t)\in L^q_\sigma(D)$ at hand, let us proceed to the next stage, in which
the solution $u(t)$ to \eqref{weak-int} is identified with a local solution to the initial
value problem 
\begin{equation}
u(t)=T(t,t_0)u(t_0)
-\int_{t_0}^tT(t,s)P(u\cdot\nabla u+b\cdot\nabla u+u\cdot\nabla b-f)(s)\, ds
\label{auxi-ivp}
\end{equation}
in a neighborhood of each instant $\tau_0\in\mathbb R$.
This idea was developed by Kozono and Yamazaki \cite{KY2}
and also adapted in \cite{Hi09}, \cite{TT}.
We use the following elementary lemma.
The reason why we employ the operator $A^{1/2}$ (as in \cite{Hi09}),
see \eqref{stokes} and \eqref{stokes-evo},
is to verify the boundary condition $u(t)|_{\partial D}=0$ in the sense 
of trace via $u(t)\in D_q(A^{1/2})\subset W^{1,q}_0(D)$.
\begin{lemma}
Let $1<q<\infty$, $0\leq \alpha <1$ and $t_0\in\mathbb R$.
Set $I=(t_0,t_0+\tau_*)$ with some $\tau_*\in (0,2]$.
Suppose that
\begin{equation}
(t-t_0)^\alpha g\in L^\infty(I;\,L^q_\sigma(D)), \quad
\|g(t)\|_q\leq k_g\,(t-t_0)^{-\alpha}\quad (\mbox{\rm a.e.}\, t\in I)
\label{force-cond-0}
\end{equation}
with some constant $k_g>0$.
Then 
\begin{equation}
u(t):=\int_{t_0}^tT(t,s)g(s)\,ds \qquad (t\in I)
\label{duha-func}
\end{equation}
is well-defined as function that satisfies
\begin{equation}
u\in C(I;\,L^r_\sigma(D)), \quad 
\|u(t)\|_r\leq Ck_g(t-t_0)^{-\alpha-(3/q-3/r)/2+1}
\label{duha-est1}
\end{equation}
\begin{equation}
A^{1/2}u\in C_w(I;\,L^\rho_\sigma(D)), \quad
\|A^{1/2}u(t)\|_\rho\leq C^\prime k_g(t-t_0)^{-\alpha-(3/q-3/\rho)/2+1/2}
\label{duha-est2}
\end{equation}
for all $t\in I$
with some constants $C=C(q,r,\alpha,D),\, C^\prime=C^\prime(q,\rho,\alpha,D)>0$, 
which are independent of $t_0\in\mathbb R$
and $\tau_*\in (0,2]$,
where $r\in [q,\infty)$ 
with $1/r>1/q-2/3$,
and $\rho\in [q,\infty)$ with $1/\rho>1/q-1/3$ as well as $\rho >3/2$.
\label{inhomo-0}
\end{lemma}

\begin{proof}
Estimate \eqref{duha-est1} is obvious by virtue of \eqref{evo1}.
To show estimate \eqref{duha-est2}, as in \cite[Lemma 3.3]{Hi09},
let us recall $\|A^{1/2}f\|_r\leq C\|\nabla f\|_r$
for every $r\in (1,\infty)$ and $f\in D_r(A)$,
see Borchers and Miyakawa \cite[Theorem 4.4]{BM90}.
Let $f\in C_{0,\sigma}^\infty(D)$, then it follows from \cite[Proposition 2.1]{Hi20} that
$T(t,s)f\in D_r(L_+(t))\subset D_r(A)$ for $t>s$ as long as $r\in (3/2,\infty)$,
where this restriction on the exponent $r$ stems from Lemma 5.2 of \cite{Hi20}.
By 
Remark \ref{rem-evo} (i) on the smoothing rate
near $t=s$, we have
\begin{equation}
\|A^{1/2}T(t,s)f\|_r\leq C\|\nabla T(t,s)f\|_r
\leq C(t-s)^{-1/2-(3/q-3/r)/2}\|f\|_q
\label{stokes-evo}
\end{equation}
for $t-s\leq \tau_*\leq 2$ and $1<q\leq r<\infty$ as well as $r>3/2$.
By closedness of $A^{1/2}$, it turns out that $T(t,s)f\in D_r(A^{1/2})$ for general $f\in L^q_\sigma(D)$ along with 
\eqref{stokes-evo}, 
which implies the estimate in \eqref{duha-est2}.
The emphasis is that the constant $C>0$ in the right-hand side
of \eqref{stokes-evo} is independent of $t,\,s$
with $t-s\leq 2$.

The strong continuity \eqref{duha-est1} of $u(t)$ with values in $L^r_\sigma(D)$,
that is not obvious for the non-autonomous case,
is due to Takahashi \cite[Lemma 4.6]{TT}.
As for the other continuity property \eqref{duha-est2},
since we know the estimate 
in \eqref{duha-est2},
it suffices to show that
$\langle A^{1/2}(u(t+h)-u(t)),\psi\rangle$
goes to zero as $h\to 0$ for 
$\psi\in C_{0,\sigma}^\infty(D)$,  
which immediatelly follows from 
$u\in C(I;\,L^r_\sigma(D))$.
The proof is complete.
\end{proof}

Since 
$q\in (3,6)$, 
the length of the existence interval of a local solution to \eqref{auxi-ivp}
with $u(t_0)\in L^q_\sigma(D)$
can be estimated from below by $\|u(t_0)\|_q$ as well as $\|(\eta,\omega)\|_{W^{1,\infty}}$,
which is uniformly bounded, see \eqref{bdd-zuerst}.
Thus, as in \cite[Proposition 3.2, Proof of (3) of Theorem 2.1]{Hi09} and
\cite[Proposition 4.7, Lemma 4.8, Proof of Theorem 2.1]{TT},
one can easily find $\gamma\in (0,1]$, which
is independent of $\tau_0$ and $(\eta,\omega)$ satisfying $\|(\eta,\omega)\|_{W^{1,\infty}}\leq\delta$,
such that 
for any $\tau_0\in\mathbb R$, 
\eqref{auxi-ivp} with $t_0=\tau_0-\gamma$
admits a solution, denoted by $\widetilde u(t)$, 
on the interval $I_{\tau_0}:=(\tau_0-\gamma, \tau_0+\gamma)$, which satisfies
\begin{equation}
\widetilde u\in C(I_{\tau_0};\,L^r_\sigma(D))), \qquad  
A^{1/2}\widetilde u\in C_w(I_{\tau_0};\,L^r_\sigma(D))
\label{local-conti}
\end{equation}
\begin{equation}
\begin{split}
(t-t_0)^{(3/q-3/r)/2}\|\widetilde u(t)\|_r+
&(t-t_0)^{1/2+(3/q-3/r)/2}\|A^{1/2}\widetilde u(t)\|_r  \\
&\leq C\big(\|u(t_0)\|_q+\|(\eta,\omega)\|_{W^{1,\infty}}\big)
\end{split}
\label{local-est}
\end{equation}
for all $t\in I_{\tau_0}$ 
and $r\in [q,\infty)$ 
by using Lemma \ref{inhomo-0} and \eqref{f-F},
where the constant $C=C(q,r,D)>0$
is independent of $\tau_0$.
With this solution $\widetilde u(t)$ at hand, we have the following proposition.
\begin{proposition}
Let $u(t)$ be the solution to \eqref{weak-int}
obtained in subsection 3.1.
Then $u(t)$ is a solution to \eqref{NS-int} in $L^r_\sigma(D)$
for every $r\in (3,\infty)$ 
possessing \eqref{sol-cl}--\eqref{reg3}.
\label{q-bdd}
\end{proposition}

\begin{proof}
Given arbitrary $\tau_0\in\mathbb R$, we already have the solution $\widetilde u(t)$ to \eqref{auxi-ivp} on
$I_{\tau_0}$.
On the other hand, the solution $u(t)$ to \eqref{weak-int} 
obtained in subsection 3.1 satisfies the weak form
\begin{equation}
\langle u(t),\psi\rangle
=\langle T(t,t_0)u(t_0),\psi\rangle
+\int_{t_0}^t\langle (u\otimes u+u\otimes b+b\otimes u-F)(s), \nabla T(t,s)^*\psi\rangle\, ds
\label{weak-ivp}
\end{equation}
for all $\psi\in C_{0,\sigma}^\infty(D)$
on the same interval $I_{\tau_0}$ bacause, by continuity, $u(t)$ enjoys \eqref{weak-int} for all 
$\psi\in L^{3/2,1}_\sigma(D)$,
so that one can replace $\psi$ by $T(t,t_0)^*\psi$ in \eqref{weak-int} with $t=t_0$ to obtain
a formula of $\langle T(t,t_0)u(t_0), \psi\rangle$.
From this we see that 
$u(t)$ actually fulfills \eqref{weak-ivp} for all $\psi\in C_{0,\sigma}^\infty(D)$.
As a consequence, both $u(t)$ and $\widetilde u(t)$ 
belong to $L^\infty(I_{\tau_0};\,L^q_\sigma(D))$ as well as enjoy \eqref{weak-ivp}.
Since the uniqueness of solutions to \eqref{weak-ivp} within this class with $q>3$ is verified as in 
\cite[Lemma 3.4]{Hi09} and 
\cite[Lemma 4.5]{TT}
(where continuity of solutions in $t$ is also assumed in the former paper, however, it is redundant from the proof as in the latter paper),
we infer that $u(t)=\widetilde u(t)$ on $I_{\tau_0}$.
Therefore, the solution $u(t)$ possesses all the properties deduced from \eqref{auxi-ivp}.
In constructing $\widetilde u(t)$ for \eqref{auxi-ivp},
we can take $q\in (3,6)$ as close to $3$ as we wish,
so we obtain \eqref{local-conti}--\eqref{local-est}
with $r\in (3,\infty)$ for $\widetilde u(t)$ replaced by $u(t)$ to lead to
$u\in C(\mathbb R;\, L^r_\sigma(D))\cap C_w(\mathbb R;\, W^{1,r}_0(D))$
for every $r\in (3,\infty)$.
Let us restrict \eqref{local-est} to the interval
$J_{\tau_0}:=(\tau_0-\frac{\gamma}{2}, \tau_0+\gamma)\subset I_{\tau_0}$;
then, by \eqref{bdd-zuerst} we get 
\begin{equation}
\|u(t)\|_r 
+\|\nabla u(t)\|_r
\leq C\|(\eta,\omega)\|_{W^{1,\infty}} \qquad (t\in J_{\tau_0})
\label{bdd-danach}
\end{equation}
for every 
$r\in (3,\infty)$ 
with some constant $C=C(r,D)>0$ 
independent of $\tau_0\in\mathbb R$.
Since $\tau_0$ is arbitrary, we conclude \eqref{reg1}--\eqref{reg3}.

We use Lemma \ref{b-s}
with the aid of \eqref{reg3} 
as well as \eqref{lift-est} to see that
the integral of the right-hand side of \eqref{NS-int}
is Bochner summable
in $L^r(D)$ for every $r\in (3,\infty]$; to be precise,
\begin{equation*}
\begin{split}
&\quad \int_{-\infty}^t
\|T(t,s)P\,\mbox{div $(u\otimes u+u\otimes b+b\otimes u)$}(s)\|_r\, ds \\
&\leq C\left(\sup_{t\in\mathbb R}\big(\|u(t)\|_{r_0}^2+\|u(t)\|_{r_1}^2\big)
+\|(\eta,\omega)\|_{W^{1,\infty}}^2\right)
\end{split}
\end{equation*}
where 
$r_0$ and $r_1$ are chosen in such a way that
$3<r_0<6r/(r+3)<r_1<2r$ for given
$r\in (3,\infty]$.
The term involving the force $f$ is harmless on account of \eqref{f-F}.
We already know that
\begin{equation}
u(t)=T(t,\tau)u(\tau)
+\int_\tau^tT(t,s)P
\big[f(s)-\mbox{div$(u\otimes u+u\otimes b+b\otimes u)$}(s)\big]\, ds
\label{NS-int-pre}
\end{equation}
in $L^r_\sigma(D)$, $3<r<\infty$,
for all $\tau,\, t$ with $-\infty <\tau <t<\infty$.
In fact, for each $\tau_0\in [\tau,t]$, we have \eqref{NS-int-pre}
with $\tau,\, t$ replaced by every pair of
$\tau^\prime,\, t^\prime\in I_{\tau_0}=(\tau_0-\gamma, \tau_0+\gamma)$ satisfying
$\tau^\prime< t^\prime$, 
where $I_{\tau_0}$ is the existence interval of $\widetilde u(t)$ constructed above.
Clearly, there are $\tau_1, \cdots, \tau_k\in [\tau,t]$
such that $[\tau,t]\subset\bigcup_{j=1}^kI_{\tau_j}$,
yielding \eqref{NS-int-pre}.
Since
\[
\|T(t,\tau)u(\tau)\|_r\leq C(t-\tau)^{-(3/q_0-3/r)/2}\|u(\tau)\|_{q_0}
\to 0\quad (\tau\to -\infty)
\]
with some $q_0\in (3,r)$ follows from \eqref{evo1} and
\eqref{reg3}, $u(t)$ satisfies
\eqref{NS-int} in $L^r_\sigma(D)$ for every $r\in (3,\infty)$.
This completes the proof.
\end{proof}

\noindent
{\bf 3.3. Regularity in $t$ and the pressure}
\medskip

In this subsection the other regularity \eqref{reg2} of the solution $u(t)$ to \eqref{NS-int}
obtained in Proposition \ref{q-bdd} is deduced 
together with recovery of the associated pressure.
We also derive the other weak form, see \eqref{weak-ext} below. 
To this end, the equation \eqref{NS-int} itself does not seem convenient.
Instead, we deal with \eqref{NS-int-pre} or \eqref{auxi-ivp}.
Let us use \eqref{auxi-ivp}, however, it is not ragarded as the initial value problem to be solved unlike
the argument in the previous subsection.
In the present context that we already have the solution $u(t)$ with \eqref{sol-cl}--\eqref{reg3}, 
we may choose $t_0$ as we wish.
Given arbitrary $\tau_0\in\mathbb R$, we simply take $t_0=\tau_0-1$, $I_{\tau_0}:=(\tau_0-1, \tau_0+1)$ and
$J_{\tau_0}:=(\tau_0-\frac{1}{2}, \tau_0+1)$.
We then intend to show the desired properties of the solution $u(t)$ on the latter interval $J_{\tau_0}$
by investigating each of \eqref{sol-sepa} below separately:  
$u(t)=v(t)+w(t)$
with
\begin{equation}
\begin{split}
&v(t):=T(t,t_0)u(t_0), \\
&w(t):=-\int_{t_0}^tT(t,s)P(u\cdot\nabla u+b\cdot\nabla u+u\cdot\nabla b-f)(s)\,ds.
\end{split}
\label{sol-sepa}
\end{equation}
We apply \cite[Proposition 5.1]{Hi20} on regularity of the evolution operator
directly to $v(t)$, while somewhat new arguments are needed for $w(t)$
with the aid of properties of the function \eqref{duha-func}.
Concerning the latter, we still invoke
\cite[Proposition 5.1]{Hi20} to show the
following lemma.
Note that one can not assume the strong continuity in \eqref{force-cond}
below because of lack of such property for $\nabla u(t)$, see \eqref{reg1},
since we have 
$g=-P(u\cdot\nabla u+b\cdot\nabla u+u\cdot\nabla b-f)$ in mind
and that, under weak regularity as in \eqref{force-cond}, the function
\eqref{duha-func} could not be of class $C^1$ in $t$ strongly with values in $L^q_\sigma(D)$ even if the evolution
operator $T(t,s)$ were of parabolic type. 
In fact, it is not of such type 
(in the sense that, for each $t\in\mathbb R$, the operator $L_+(t)$
is not a generator of an analytic semigroup unless $\omega(t)=0$
\cite{Hi99}) and thus the situation is even harder.
On the other hand, 
we have no longer singular behavior like \eqref{force-cond-0} at the end
point of any bounded interval in $\mathbb R$ 
when we keep the same $g$ as above in mind
because of \eqref{reg1}--\eqref{reg3}. 
\begin{lemma}
Let $1<q<\infty$, $R> 2$ and $t_0\in\mathbb R$. 
Set $I=(t_0,t_0+\tau_*)$ with some $\tau_*\in (0,2]$.
Suppose that
\begin{equation}
g\in BC_w(I;\,L^q_\sigma(D)), \qquad
\|g(t)\|_q\leq k_g 
\quad (t\in I)
\label{force-cond}
\end{equation}
with some constant $k_g>0$.
Then the function $u(t)$ given by \eqref{duha-func} is of class
\begin{equation}
u\in C^1_{w}(I;\,W^{-1,q}(D_R)) 
\label{weak-diff}
\end{equation}
and obeys
\begin{equation}
\langle \partial_tu,\psi\rangle
+\langle \nabla u+u\otimes (\eta+\omega\times x), \nabla\psi\rangle
+\langle \omega\times u, \psi\rangle
-\langle p, \mbox{\rm{div} $\psi$}\rangle
=\langle g,\psi\rangle
\label{weak-inhomo}
\end{equation}
for all $\psi\in W_0^{1,q^\prime}(D_R)^3$ and $t\in I$ together with the associated pressure $p(t)$ over $D$,
which is singled out in such a way that
$\int_{D_R}p(t)\,dx=0$ and satisfies
\begin{equation}
p\in C_w(I;\,L^q(D_R)).
\label{pre-cl-ivp}
\end{equation}
Moreover, there is a constant $C=C(q,R,D)>0$ independent of $t_0\in\mathbb R$
and $\tau_*\in (0,2]$ such that 
\begin{equation}
\|\partial_tu(t)\|_{W^{-1,q}(D_R)}
+\|p(t)\|_{q,D_R}
\leq Ck_g
\label{pre-ivp}
\end{equation}
for all $t\in I$.
\label{weak-c1}
\end{lemma}

\begin{proof}
We fix $R> 2$ arbitrarily.
Let $t_0<s<t<t_0+\tau_*$.
Then, for each $s$, it follows from \cite[Proposition 5.1]{Hi20}
and \eqref{force-cond} that
\begin{equation}
v(\cdot,s):=T(\cdot,s)g(s)\in C^1((s,t_0+\tau_*);\,W^{-1,q}(D_R)) 
\label{integrand-c1}
\end{equation}
together with
\begin{equation}
\|\partial_tv(t,s)\|_{W^{-1,q}(D_R)}
\leq Ck_g(t-s)^{-(1+1/q)/2}, 
\qquad t\in (s,t_0+\tau_*)
\label{integrand-est}
\end{equation}
and that 
\begin{equation}
\langle \partial_tv,\psi\rangle
+\langle\nabla v+v\otimes (\eta(t)+\omega(t)\times x), \nabla\psi\rangle
+\langle \omega(t)\times v, \psi\rangle =0
\label{integrand}
\end{equation}
for all 
$\psi\in W_{0,\sigma}^{1,q^\prime}(D_R):=\{\psi\in W_0^{1,q^\prime}(D_R);\, \mbox{div $\psi$}=0\}$ 
and $t\in (s,t_0+\tau_*)$.
It should be emphasized that the constant $C=C(q,R,D)>0$
in \eqref{integrand-est} is
independent of 
$t_0\in\mathbb R$ and $\tau_*\in (0,2]$ as well as $t,\, s\in I$.
The weak form (5.21) of \cite{Hi20} involves the pressure, but that is not the case
in \eqref{integrand} since the test function $\psi$ is
in particular solenoidal.
Another remark is the relation
$\langle \omega\times v, \psi\rangle=-\langle (\omega\times x)\otimes v, \nabla\psi\rangle$,
whose right-hand side appears in \cite[(5.21)]{Hi20}, nevertheless, it is convenient for later use
to keep the Coriolis term as it is.

Let $\psi\in W^{1,q^\prime}_0(D_R)\subset W^{1,q^\prime}(D)$ (extension by setting zero outside $D_R$
so that $P\psi$ below makes sense), then
\[
(t_0,t)\ni s\mapsto \langle v(t,s), \psi\rangle
=\langle T(t,s)g(s),\psi\rangle
\]
is continuous 
on account of \eqref{force-cond}; in fact, we observe
\begin{equation*}
\begin{split}
&\quad \langle T(t,s+h)g(s+h),\psi\rangle
-\langle T(t,s)g(s),\psi\rangle  \\
&=\langle g(s+h), \big(T(t,s+h)^*-T(t,s)^*\big)P\psi\rangle
+\langle g(s+h)-g(s), T(t,s)^*P\psi\rangle \to 0
\end{split}
\end{equation*}
as $h\to 0$.
This combined with \eqref{integrand-c1} implies that
$\langle u(\cdot),\psi\rangle \in C^1(I)$ 
for every 
$\psi\in W^{1,q^\prime}_0(D_R)$
along with
\begin{equation}
\frac{d}{dt}\langle u(t),\psi\rangle
=\langle g(t),\psi\rangle
+\int_{t_0}^t \langle\partial_t v(t,s),\psi\rangle\,ds, \qquad\forall \psi\in W_0^{1,q^\prime}(D_R).
\label{diff-in-between}
\end{equation}
By \eqref{force-cond}, \eqref{integrand-est} and \eqref{diff-in-between},
there is a constant $C=C(q,R,D)>0$ such that
\[
\left|\frac{d}{dt}\langle u(t),\psi\rangle\right|
\leq C
k_g\|\nabla \psi\|_{q^\prime,D_R}
\]
for all $\psi\in W^{1,q^\prime}_0(D_R)$ and $t\in I$, from which we conclude \eqref{weak-diff}
and \eqref{pre-ivp} for $\partial_tu(t)$.  

If we take, in particular, $\psi\in W^{1,q^\prime}_{0,\sigma}(D_R)$ in
\eqref{diff-in-between}, then \eqref{integrand} implies
\begin{equation}
\langle\partial_tu, \psi\rangle
+\langle\nabla u+u\otimes (\eta+\omega\times x), \nabla\psi\rangle
+\langle \omega\times u, \psi\rangle 
=\langle g,\psi\rangle
\label{weak-inhomo-pre}
\end{equation}
for such $\psi$.
Set
\[
{\mathcal I}[u](t):=\int_{t_0}^tu(\tau)\,d\tau
\]
that is well-defined in $L^q_\sigma(D)$ 
by \eqref{duha-est1} with $r=q$ and $\alpha =0$.
Likewise, we define
\[
{\mathcal I}[u\otimes (\eta+\omega\times x)](t),\quad  
{\mathcal I}[\omega\times u](t),\quad  
{\mathcal I}[g](t)
\]
by the Bochner integrals of those functions over the interval $(t_0,t)$,
which are well-defined in 
$L^q(D_R),\;L^q(D)$ and $L^q_\sigma(D)$, respectively.
By \eqref{weak-inhomo-pre} we then deduce
\begin{equation}
\langle u,\psi\rangle+\langle\nabla {\mathcal I}[u]+{\mathcal I}[u\otimes (\eta+\omega\times x)], \nabla\psi\rangle
+\langle {\mathcal I}[\omega\times u],\psi\rangle
=\langle {\mathcal I}[g],\psi\rangle
\label{weak-auxi}
\end{equation}
for all $\psi\in W^{1,q^\prime}_{0,\sigma}(D_R)$.
Hence, there exists a distribution 
${\mathcal P}={\mathcal P}(t)$ 
such that
\begin{equation}
\langle u,\psi\rangle
+\langle \nabla {\mathcal I}[u]+{\mathcal I}[u\otimes (\eta+\omega\times x)], \nabla\psi\rangle
+\langle {\mathcal I}[\omega\times u], \psi\rangle
-\langle {\mathcal I}[g],\psi\rangle
=\langle {\mathcal P},\mbox{div $\psi$}\rangle
\label{weak-auxi2}
\end{equation}
for all 
$\psi\in C_0^\infty(D_R)$.
Since the left-hand side of \eqref{weak-auxi2} belongs to $C^1(I)$,
one can define $Q(t)$ by 
\begin{equation*}
\langle Q(t),\psi\rangle
:=\frac{d}{dt}\langle {\mathcal P}(t),\mbox{div $\psi$}\rangle
=\langle\partial_tu,\psi\rangle
+\langle\nabla u+u\otimes (\eta+\omega\times x), \nabla\psi\rangle
+\langle \omega\times u, \psi\rangle 
-\langle g,\psi\rangle  
\end{equation*}
for all $\psi\in C_0^\infty(D_R)$ and, by continuity, for 
all $\psi\in W_0^{1,q^\prime}(D_R)$
to infer that $Q(t)\in W^{-1,q}(D_R)$ with
\[
\|Q(t)\|_{W^{-1,q}(D_R)}\leq c_0k_g
\]
on account of \eqref{eta-om0}, \eqref{duha-est1}--\eqref{duha-est2} with $r=\rho=q$
and $\alpha =0$,
\eqref{force-cond} and \eqref{pre-ivp} for $\partial_tu(t)$,
where the constant $c_0=c_0(q,R,D)>0$ is independent of 
$t_0\in\mathbb R$ and $\tau_*\in (0,2]$.
If in particular $\mbox{div $\psi$}=0$, then
$\langle Q(t), \psi\rangle=0$.
Hence,
we find $p_{_R}=p_{_R}(t)\in L^q(D_R)$ subject to 
$\int_{D_R}p_{_R}(t)\,dx=0$ 
such that $Q(t)=-\nabla p_{_R}(t)$ 
together with
\[
\|p_{_R}(t)\|_{q,D_R}\leq C\|Q(t)\|_{W^{-1,q}(D_R)}
\leq Cc_0k_g
\]
for all $t\in I$, see Sohr \cite[Chapter II, Lemma 2.1.1]{So},
and that
\eqref{weak-inhomo} holds with $p=p_{_R}$ for all 
$\psi\in W^{1,q^\prime}_0(D_R)$ and $t\in I$.

Given $\phi\in L^{q^\prime}(D_R)$, there is a function
$\psi\in W_0^{1,q^\prime}(D_R)$
that satisfies
\[
\mbox{div $\psi$}=\phi-\frac{1}{|D_R|}\int_{D_R}\phi\,dx\quad\mbox{in $D_R$}.
\]
In fact, 
we may for instance take a particular solution discovered by Bogovskii \cite{B, BS, G-b}.
We then use the relation \eqref{weak-inhomo} together with
$\int_{D_R}(p_{_R}(t+h)-p_{_R}(t))\,dx=0$, \eqref{eta-om0},
\eqref{duha-est1}--\eqref{duha-est2} and \eqref{force-cond}--\eqref{weak-diff}
to observe
\[
\langle p_{_R}(t+h)-p_{_R}(t), \phi\rangle
=\langle p_{_R}(t+h)-p_{_R}(t), \mbox{div $\psi$}\rangle
\to 0\quad (h\to 0),
\]
which concludes $p_{_R}\in C_w(I; L^q(D_R))$.

It remains to construct a pressure defined over the whole $D$.
By the same procedure as above, for every integer $k>0$ one can obtain the pressure $p_{_{R+k}}$ over $D_{R+k}$
such that all the properties deduced above with $R$ replaced by $R+k$ are available.
Then it turns out from \eqref{weak-inhomo} that
\[
\langle p_{_{R+k}}(t)-p_{_{R+j}}(t), \mbox{div $\psi$}\rangle=0
\]
for every $\psi\in C_0^\infty(D_{R+j})^3$ and $k>j\geq 0$.
As a consequence, $p_{_{R+k}}-p_{_R}=c_k(t)$ with some $c_k(t)\in\mathbb R$ that is dependent only on $t$.
Let us define
\[
p(x,t):=
\left\{
\begin{array}{ll}
p_{_R}(x,t), &x\in D_R, \\
p_{_{R+k}}(x,t)-c_k(t), \qquad &x\in D_{R+k}\setminus D_{R+k-1}\quad (k=1,2,\cdots),
\end{array}
\right.
\]
which is the desired pressure over $D$. 
The proof is complete.
\end{proof}

The remaining part of Theorem \ref{main1} and Corollary \ref{per-sol} are established in the
following proposition.
\begin{proposition}
Let $u(t)$ be the solution to \eqref{NS-int} obtained in Proposition \ref{q-bdd}.
Given arbitrary $R>2$, we have
\begin{equation}
u\in C^1_{w}(\mathbb R;\,W^{-1,r}(D_R)) 
\label{weak-d}
\end{equation}
for all $r\in (1,\infty)$ as well as
\begin{equation}
\begin{split}
\langle \partial_t u,\psi\rangle
+\langle \nabla u+u\otimes (\eta+\omega\times x), 
&\nabla\psi\rangle
+\langle \omega\times u, \psi\rangle
-\langle p, \mbox{\rm{div} $\psi$}\rangle  \\
&=\langle u\otimes u+u\otimes b+b\otimes u, \nabla\psi\rangle
+\langle f,\psi\rangle
\end{split}
\label{weak-ext}
\end{equation}
for all $\psi\in W_0^{1,r^\prime}(D_R)^3$ 
and $t\in\mathbb R$ together with the associated pressure $p(t)$ over $D$,
which is singled out in such a way that
$\int_{D_R}p(t)\,dx=0$ and satisfies
\begin{equation}
p\in BC_w(\mathbb R;\, L^r(D_R))
\label{pre-cl}
\end{equation}
along with \eqref{reg4}.

If in particular $\eta$ and $\omega$ are $l$-periodic with some period $l>0$,
then the conclusion of Corollary \ref{per-sol} is true.
\label{weak-exterior}
\end{proposition}

\begin{proof}
Let us fix $R> 2$ arbitrarily.
For the first falf (before periodic issue)
it suffices to show the conclusion 
for $r\in (3,\infty)$.
We also fix such exponent $r$ in what follows.
Following what is described at the beginning of this subsection,
we take $\tau_0\in\mathbb R$, set
$I_{\tau_0}=(t_0,\tau_0+1)$ with $t_0=\tau_0-1$
and begin to discuss $v(t)$ given by \eqref{sol-sepa}.
According to \cite[Proposition 5.1]{Hi20}, there is a function $p_v(t)=p_{v,\tau_0}(t)\in L^r_{\rm loc}(\overline{D})$
such that the pair $\{v,p_v\}$ satisfies
\begin{equation}
v\in C^1(I_{\tau_0};\,W^{-1,r}(D_R))
\label{v-1}
\end{equation}
\begin{equation}
p_v\in C(I_{\tau_0};\,L^r(D_R)), \qquad \int_{D_R}p_v(t)\,dx=0
\label{p-v}
\end{equation}
\begin{equation}
\|\partial_t v(t)\|_{W^{-1,r}(D_R)}+\|p_v(t)\|_{r,D_R}
\leq C(t-t_0)^{-(1+1/r)/2}\|u(t_0)\|_r  
\label{v-2}
\end{equation}
\begin{equation}
\langle\partial_tv,\psi\rangle
+\langle\nabla v+v\otimes (\eta+\omega\times x), \nabla\psi\rangle
+\langle \omega\times v, \psi\rangle
-\langle p_v,\mbox{div $\psi$}\rangle=0
\label{v-3}
\end{equation}
for all $\psi\in W^{1,r^\prime}_0(D_R)$ and 
$t\in I_{\tau_0}=(t_0,\tau_0+1)$ with $t_0=\tau_0-1$,
where the constant $C=C(r,R,D)>0$ in \eqref{v-2} is independent of $\tau_0$
and the continuity \eqref{p-v} is not given in the statement
of Proposition 5.1 of \cite{Hi20} but found at the very end of its proof
in that paper.
We note that \eqref{v-1}--\eqref{p-v} and \eqref{v-3}
actually hold on the interval $(t_0,\infty)$.

We turn to $w(t)$ given by \eqref{sol-sepa}.
By \eqref{lift-est}, \eqref{f-F} and \eqref{reg1}--\eqref{reg3},
the function
\[
g:=-P(u\cdot\nabla u+b\cdot\nabla u+u\cdot\nabla b-f)
\]
satisfies \eqref{force-cond} with $q=r\in (3,\infty)$ and
\begin{equation*}
\begin{split}
k_g&=C\sup_{t\in I}\left(\|u(t)\|_{2r}^2+\|\nabla u(t)\|_{2r}^2\right)
+C\|(\eta,\omega)\|_{W^{1,\infty}}^2+C\|(\eta,\omega)\|_{W^{1,\infty}}  \\
&\leq C\|(\eta,\omega)\|_{W^{1,\infty}}
\end{split}
\end{equation*}
for every bounded interval $I\subset \mathbb R$.
One can thus apply Lemma \ref{weak-c1} to $w(t)$
on the interval $I=I_{\tau_0}=(t_0, \tau_0+1)$ with $t_0=\tau_0-1$.
We then find a function $p_w(t)=p_{w,\tau_0}(t)\in L^r_{\rm loc}(\overline{D})$ such that the pair $\{w,p_w\}$ satisfies
\begin{equation}
w\in C^1_{w}(I_{\tau_0};\, W^{-1,r}(D_R))
\label{w-1}
\end{equation}
\begin{equation}
p_w\in C_w(I_{\tau_0};\,L^r(D_R)), \qquad \int_{D_R}p_w(t)\,dx=0
\label{p-w}
\end{equation}
\begin{equation}
\|\partial_tw(t)\|_{W^{-1,r}(D_R)} 
+\|p_w(t)\|_{r,D_R}\leq 
C\|(\eta,\omega)\|_{W^{1,\infty}}
\label{w-0}
\end{equation}
\begin{equation}
\begin{split}
\langle\partial_tw,\psi\rangle
+\langle\nabla w+w\otimes (\eta+\omega\times x),
&\nabla\psi\rangle
+\langle \omega\times w, \psi\rangle
-\langle p_w,\mbox{div $\psi$}\rangle  \\
&=\langle u\otimes u+u\otimes b+b\otimes u, \nabla\psi\rangle+\langle f,\psi\rangle
\end{split}
\label{w-2}
\end{equation}
for all $\psi\in W^{1,r^\prime}_0(D_R)$ and
$t\in I_{\tau_0}=(t_0,\tau_0+1)$ with $t_0=\tau_0-1$,
where the constant $C=C(r,R,D)>0$ in \eqref{w-0} is independent of $\tau_0$.

By \eqref{v-1} and \eqref{w-1} for every $\tau_0\in \mathbb R$ we immediately see that $u=v+w$
enjoys \eqref{weak-d}.
Since the associated pressure $p^{\tau_0}:=p_{v,\tau_0}+p_{w,\tau_0}$
is available 
solely on each interval $I_{\tau_0}$ and \eqref{weak-ext} is satisfied for $\{u,p^{\tau_0}\}$
only there,
we need to construct a pressure globally
on the whole time axis $\mathbb R$.
Once we have that, it follows from \eqref{v-3} and \eqref{w-2} that
\eqref{weak-ext} holds true for every $t\in\mathbb R$.

If $I_{\tau_0}\cap I_{\tau_0^\prime}\neq \emptyset$, then we use \eqref{weak-ext} locally to see that
$\nabla (p^{\tau_0}(t)-p^{\tau_0^\prime}(t))=0$ for $t\in I_{\tau_0}\cap I_{\tau_0^\prime}$, which together with
$\int_{D_R}(p^{\tau_0}(t)-p^{\tau_0^\prime}(t))\,dx=0$ implies that both pressures coincide with each other
in the intersection of those intervals.
In this way, the pressure, denoted by $p(t)$, is well-defined on the whole line $\mathbb R$.
By \eqref{p-v} and \eqref{p-w}, we find $p\in C_w(\mathbb R;\, L^r(D_R))$.
By \eqref{v-2} especially on the interval
$J_{\tau_0}=(\tau_0-\frac{1}{2},\tau_0+1)\subset I_{\tau_0}$, \eqref{reg3} and
\eqref{w-0}, there is a constant $C=C(r,R,D)>0$ independent of $\tau_0\in\mathbb R$ such that
\[
\|\partial_tu(t)\|_{W^{-1,r}(D_R)}
+\|p(t)\|_{r,D_R}
\leq C\|(\eta,\omega)\|_{W^{1,\infty}} \qquad   
(t\in J_{\tau_0}).
\]
Since $\tau_0$ is arbitrary, we conclude 
$p\in BC_w(\mathbb R;\, L^r(D_R))$ together with \eqref{reg4}.

Suppose, in particular, that $\eta$ and $\omega$ are $l$-periodic, then we have \eqref{per-evo-op}.
This together with the periodicity of $b$ and $f$ implies that $u(t+l)$ enjoys the same equation \eqref{NS-int}
within the same class \eqref{sol-cl}.
By uniqueness of solutions in
Theorem \ref{main1}, see Remark \ref{sol-bdd} (i),
we conclude $u(t+l)=u(t)$ for every $t\in\mathbb R$.
Furthermore,
the relation \eqref{weak-ext} for all $t\in\mathbb R$ shows that
the associated pressure $p(t)$ constructed above 
satisfies $\nabla (p(t+l)-p(t))=0$, which along with
$\int_{D_R}(p(t+l)-p(t))\,dx=0$ leads to $p(t+l)=p(t)$ for every $t\in\mathbb R$.
The proof is complete.
\end{proof}

\section{Proof of Theorem \ref{main3}}
\label{pr3}

\noindent
{\bf 4.1.
Reduction to the whole space problem}
\medskip

As decribed in section \ref{intro}, we carry out a cut-off procedure to reduce
our consideration to the whole space problem.
We take a cut-off function $\varphi$ as in \eqref{cut} and then use the 
Bogovskii operator
$\mathbb B$ in the bounded domain
$B_2\setminus \overline{B_1}$,
see \cite{B, BS, G-b, GHH} (brief descriptions given in \cite[p.220]{Hi20} are enough for our purpose),
to recover the solenoidal condition.
Let $u(t)$ be the solution obtained in Theorem \ref{main1} and
set
\begin{equation}
v(t)=(1-\varphi)u(t)+u_c(t), \quad u_c(t)=\mathbb B[u(t)\cdot\nabla\varphi].
\label{modi0}
\end{equation}
Note that
$\int_{B_2\setminus \overline{B_1}}u\cdot\nabla\varphi\,dx=\int_{\partial D}\nu\cdot u\,d\sigma=0$,
yielding $\mbox{div $v$}=0$,
and that $u_c(t)\in W^{2,q}_0(B_2\setminus \overline{B_1})\subset W^{2,q}(\mathbb R^3)$
for $q\in (3,\infty)$, where $u_c(t)$
is understood as its extension by setting zero outside $B_2\setminus\overline{B_1}$.
It follows from optimal regularity properties (and consequences from them by interpolation)
of the Bogovskii operator found in the literature above
together with \eqref{sol-cl}--\eqref{reg4}
that
\begin{equation}
\begin{split}
&u_c\in BC_{w^*}(\mathbb R;\,L^{3,\infty}(\mathbb R^3))\cap BC(\mathbb R;\, W^{1,q}(\mathbb R^3)) \\
&\sup_{t\in\mathbb R}\|u_c(t)\|_{3,\infty,\mathbb R^3}\leq C\|(\eta,\omega)\|_{W^{1,\infty}}
\end{split}
\label{correction0}
\end{equation}
\begin{equation}
\begin{split}
&u_c\in C^1_w(\mathbb R;\,L^r(B_R)), \quad \partial_tu_c(t)=\mathbb B[\partial_tu(t)\cdot\nabla \varphi] \\
&\|\partial_tu_c(t)\|_{r,B_R}\leq C^\prime\|(\eta,\omega)\|_{W^{1,\infty}}  
\end{split}
\label{correction}
\end{equation}
for all $q\in (3,\infty)$, $r\in (1,\infty)$ and $R> 2$ with some
constants $C=C(D)>0$ and $C^\prime=C^\prime(r,R,D)>0$.
We only show the weak differentiability in \eqref{correction} since the others are observed easily.
The adjoint operator 
$\mathbb B^*: L^{r^\prime}(B_2\setminus \overline{B_1})^3\to W^{1,r^\prime}(B_2\setminus\overline{B_1})$ is bounded,
from which together with \eqref{reg2} it follows that
\begin{equation*}
\begin{split}
&\quad \left\langle\frac{u_c(t+h)-u_c(t)}{h}-\mathbb B[\partial_tu(t)\cdot\nabla\varphi],\;\psi\right\rangle  \\
&=\left\langle\frac{u(t+h)-u(t)}{h}-\partial_tu(t),\;(\nabla\varphi)\mathbb B^*\psi_0
\right\rangle
\to 0 \qquad (h\to 0)
\end{split}
\end{equation*}
for every $\psi\in L^{r^\prime}(B_R)^3$, where $\psi_0:=\psi|_{B_2\setminus\overline{B_1}}$.
Let us collect several properties of $v(t)$ in
the following proposition.
\begin{proposition}
Let $u(t)$ be the solution obtained in Theorem \ref{main1}
and $v(t)$ the function given by \eqref{modi0}.
Then there is a constant $C=C(D)>0$ such that
\begin{equation}
v\in BC_{w^*}(\mathbb R;\,L^{3,\infty}_\sigma(\mathbb R^3)), \quad
\sup_{t\in\mathbb R} \|v(t)\|_{3,\infty,\mathbb R^3}\leq C\|(\eta,\omega)\|_{W^{1,\infty}}.
\label{v-cl1}
\end{equation}
Moreover, we have
\begin{equation}
v\in BC(\mathbb R;\, L^q_\sigma(\mathbb R^3))
\cap C^1_{w}(\mathbb R;\,W^{-1,r}(B_R)) 
\label{v-cl2}
\end{equation}
\begin{equation}
\nabla v\in BC_w(\mathbb R;\, L^q(\mathbb R^3))
\label{v-cl3}
\end{equation}
for all $q\in (3,\infty)$, $r\in (1,\infty)$ and $R>2$
together with
\begin{equation}
\frac{d}{dt}\langle v,\psi\rangle
+\langle \nabla v+v\otimes (\eta+\omega\times x), \nabla\psi\rangle 
+\langle \omega\times v, \psi\rangle
=\langle v\otimes v,\nabla\psi\rangle+\langle g,\psi\rangle
\label{weak-wh}
\end{equation}
for all $\psi\in C^\infty_{0,\sigma}(\mathbb R^3)$ and $t\in\mathbb R$, where $g$ is given by
\begin{equation}
\begin{split}
g&=(1-\varphi)(f-u\cdot\nabla b-b\cdot\nabla u)
+2\nabla\varphi\cdot\nabla u+[\Delta\varphi+(\eta+\omega\times x)\cdot\nabla\varphi]\,u  \\
&\quad -(\nabla\varphi)p
+\partial_tu_c-\Delta u_c-(\eta+\omega\times x)\cdot\nabla u_c+\omega\times u_c  \\
&\quad -(u\cdot\nabla\varphi)u
-\mbox{\rm{div}$\big[\varphi(1-\varphi)u\otimes u-(1-\varphi)(u\otimes u_c+u_c\otimes u)-u_c\otimes u_c\big]$}
\end{split}
\label{g}
\end{equation}
with $p(t)$ being the pressure associated with $u(t)$.
\label{weak-whole}
\end{proposition}

\begin{proof}
Before the proof, it should be noted that the term $\langle g,\psi\rangle$ in \eqref{weak-wh}
does not change, no matter which pressure we may choose (up to functions dependent only on $t$).
This is because $\langle\nabla\varphi, \psi\rangle=0$ follows from $\mbox{div $\psi$}=0$.

By \eqref{sol-cl}--\eqref{reg4} together with 
\eqref{correction0}--\eqref{correction}
we observe \eqref{v-cl1}--\eqref{v-cl3}.
Now, given $\psi\in C^\infty_{0,\sigma}(\mathbb R^3)$, we fix $R> 2$ such that 
the support of $\psi$ in contained in $B_R$ and then single out the pressure $p(t)$ satisfying 
$\int_{D_R}p(t)\,dx=0$.
Let us take $(1-\varphi)\psi$ as the test function
in \eqref{weak-ext} 
to arrive at 
\eqref{weak-wh} by 
elementary computations.
\end{proof}

We fix $R> 2$ and take the pressure $p(t)$ in such a way that
$\int_{D_R}p(t)\,dx=0$.
We need the following properties of $g(t)$.
\begin{lemma}
Let $u(t)$ be the solution obtained in Theorem \ref{main1}.
Then the function $g(t)$ given by \eqref{g} satisfies
\begin{equation}
\begin{split}
&g\in L^\infty(\mathbb R;\, L^r(\mathbb R^3)), \qquad
g(x,t)=0\;\;\mbox{\rm{a.e.}$(\mathbb R^3\setminus B_2)\times \mathbb R$}, \\
&\|g\|_{L^\infty(\mathbb R;\, L^r(\mathbb R^3))}
\leq C\|(\eta,\omega)\|_{W^{1,\infty}}  \\
\end{split}
\label{v-g-est}
\end{equation}
for all $r\in [1,\infty)$ with some constant $C=C(r,D)>0$.
There is also a function $G\in L^\infty(\mathbb R;\, L^{3/2,\infty}(\mathbb R^3))$ such that
\begin{equation}
g=\mbox{\rm{div} $G$}, \qquad
\|G\|_{L^\infty(\mathbb R;\, L^{3/2,\infty}(\mathbb R^3))}
\leq C\|(\eta,\omega)\|_{W^{1,\infty}}
\label{G-est}
\end{equation}
with some constant $C=C(D)>0$.
\label{v-g}
\end{lemma}

\begin{proof}
All the assertions in \eqref{v-g-est} 
readily follow from \eqref{eta-om0}, \eqref{lift-est}, \eqref{f-F}, 
\eqref{reg3}, \eqref{reg4} and \eqref{correction0}--\eqref{correction}
together with estimates of the Bogovskii operator.  
The function $g$ is of the form $g=g_0+\mbox{div $G_1$}$ in \eqref{g}, where
the desired estimate of $G_1$ is implied by \eqref{sol-cl} and \eqref{correction0}.
Furthermore, we
have $g_0=\mbox{div $G_0$}$ with $G_0(t)=-\nabla(4\pi |x|)^{-1}*g_0(t)\in L^{3/2,\infty}(\mathbb R^3)$,
which also enjoys the desired estimate
on account of
$\|g_0\|_{L^\infty(\mathbb R;\, L^1(\mathbb R^3))}
\leq C\|(\eta,\omega)\|_{W^{1,\infty}}$ and
$\|G_0(t)\|_{3/2,\infty,\mathbb R^3}\leq C\|g_0(t)\|_{1,\mathbb R^3}$.
This completes the proof of \eqref{G-est}.
\end{proof}
\medskip

\noindent
{\bf 4.2.
Integral equation for the whole space problem}
\medskip

In order to analyze the spatial behavior at infinity of $v(t)$ given by \eqref{modi0},
the only clue would be the explicit representation of the Duhamel term by use of the evolution operator 
$\{U(t,s);\; -\infty<s\leq t<\infty\}$,
see \eqref{formula}--\eqref{funda-sol} in the next subsection.
The evolution operator $U(t,s)$ 
provides a solution to the initial value problem for the linearized system
\eqref{linearized} in the whole space $\mathbb R^3\times (s,\infty)$.
See \cite{CM, GHa, Ha, HR11, HR} 
and \cite[Subsection 3.2]{Hi18}, \cite[Section 3]{Hi20} for the details.
The same $L^q$-$L^r$ estimates for all $t,\,s\in\mathbb R^3$ with $t>s$ as in \eqref{evo1}--\eqref{evo5}
hold true concerning $U(t,s)$ as well;
in fact, one needs neither the restriction $r\leq 3$ nor $q\geq 3/2$ even for decay estimates
of $\nabla U(t,s)$ and $U(t,s)P_{\mathbb R^3}\,\mbox{div}$, respectively.

If we formally consider the weak form \eqref{weak-wh} with $\psi$ replaced by $U(t,s)^*\psi$,
then we are led to the integral equation, that $v(t)$ obeys, in terms of $U(t,s)$, where $U(t,s)^*$
stands for the adjoint evolution operator that provides a solution to the backward problem
for the adjoint system subject to the final condition at $t$, see \cite[Lemma 3.1]{Hi18}; that is,
as in \eqref{back},
\begin{equation}
-\partial_sU(t,s)^*\psi+L_{\mathbb R^3}(s)^*U(t,s)^*\psi=0, \quad s\in (-\infty,t);\qquad
U(t,t)\psi=\psi
\label{back-wh}
\end{equation}
with $L_{\mathbb R^3}(t)$ being the generator of $U(t,s)$, which is defined as in \eqref{L}
with obvious change, and $L_{\mathbb R^3}(t)^*$ being its adjoint,
see \eqref{L-adj} for the exterior problem and the same thing holds true for the whole space problem, too. 
The justification of the procedure above is, however, by no means obvious.
In fact, look at the term $\langle v\otimes (\omega\times x), \nabla\psi\rangle$ in \eqref{weak-wh},
then we see that the class of test functions does not extend to the 
completion of $C_{0,\sigma}^\infty(\mathbb R^3)$ in a standard Sobolev space
of first order because we have no information about $v\otimes (\omega\times x)$. 
To overcome this difficulty, we use the regularity property $U(t,s)^*\psi\in Z_q(\mathbb R^3)$
for all $\psi\in C_{0,\sigma}^\infty(\mathbb R^3)$ and $q\in (1,\infty)$
(this property is enough for our aim although we know even more, see \cite[Lemma 3.1, assertion 3]{Hi18},
\cite[Lemma 3.1, assertion 3]{Hi20}), 
where
\[
Z_q(\mathbb R^3)
:=\{u\in L^q_\sigma(\mathbb R^3)\cap W^{1,q}(\mathbb R^3);\; |x|\nabla u\in L^q(\mathbb R^3)\}
\]
is a Banach space endowed with norm
\[
\|u\|_{Z_q(\mathbb R^3)}:=\|u\|_{q,\mathbb R^3}+\|\nabla u\|_{q,\mathbb R^3}+\big\||x|\nabla u\big\|_{q,\mathbb R^3}.
\]
The following density property is thus needed.
To our knowledge, 
it is not found in the existing literature and might be useful in some other studies. 
\begin{lemma}
Let $1<q<\infty$.
Then $C^\infty_{0,\sigma}(\mathbb R^3)$ is dense in $Z_q(\mathbb R^3)$.
\label{density}
\end{lemma}
\begin{proof}
We adapt the approach developed by Masuda \cite[Appendix]{Ma}, Kozono and Sohr \cite[Lemma 4.2]{KS}
to our circumstance.
Let $k\in\mathbb N$ and $\varepsilon >0$.
Given $\psi\in Z_q(\mathbb R^3)$, we set
\[
\psi_{k,\varepsilon}:=(-\Delta+\nabla\mbox{div})\varphi_k\Big(\frac{1}{k^2}-\Delta\Big)^{-1}
(\rho_\varepsilon*\psi) \in C_{0,\sigma}^\infty(\mathbb R^3)
\]
where $\varphi_k$ is defined by
$\varphi_k=\varphi(\frac{\cdot}{k})$ with $\varphi$ being \eqref{cut}
and $\rho_\varepsilon$ stands for the standard mollifier.
Then the desired density follows from
\begin{equation}
\lim_{k\to \infty}\|\psi_{k,\varepsilon}-\rho_\varepsilon *\psi\|_{Z_q(\mathbb R^3)}=0 \quad \mbox{for every $\varepsilon >0$}
\label{appro1}
\end{equation}
\begin{equation}
\lim_{\varepsilon\to 0}\|\rho_\varepsilon *\psi-\psi\|_{Z_q(\mathbb R^3)}=0.
\label{appro2}
\end{equation}
We will discuss those convergence properties merely
in the norm $\big\||x|\nabla (\cdot)\big\|_q$
since the ones in $W^{1,q}(\mathbb R^3)$ are easier.
The specific balance between the cut-off function $\varphi_k$ and the resolvent
$(\lambda -\Delta)^{-1}$ with $\lambda =1/k^2$ is important to show \eqref{appro1}
and this is a point that does not appear in the literature such as \cite{KS,Ma}.

We begin to verify
\begin{equation}
\lim_{\varepsilon \to 0}\big\||x|\nabla(\rho_\varepsilon *\psi-\psi)\big\|_{q,\mathbb R^3}=0.
\label{appro2w}
\end{equation}
Using the relation
\[
|x|\nabla(\rho_\varepsilon *\psi-\psi)
=\int_{\mathbb R^3}(|x|-|y|)\rho_\varepsilon(x-y)\nabla\psi(y)\, dy
+\rho_\varepsilon *(|\cdot|\nabla\psi)-|x|\nabla\psi
\]
the first term of which is estimated in $L^q(\mathbb R^3)$ from above by
\[
\|(|\cdot|\rho_\varepsilon)*\nabla\psi\|_{q,\mathbb R^3}
\leq \varepsilon\|\nabla \psi\|_{q,\mathbb R^3},
\]
we observe \eqref{appro2w}.

We next show
\begin{equation}
\lim_{k\to\infty}\big\||x|\nabla (\psi_{k,\varepsilon}-\rho_\varepsilon *\psi)\big\|_{q,\mathbb R^3}=0, \qquad
\forall\varepsilon >0.
\label{appro1w}
\end{equation}
To this end, let us recall the asymptotic behavior of the resolvent
\begin{equation}
\lambda\|(\lambda-\Delta)^{-1}f\|_{q,\mathbb R^3}
+\sqrt{\lambda}\|\nabla (\lambda-\Delta)^{-1}f\|_{q,\mathbb R^3}\to 0 \qquad (0<\lambda\to 0)
\label{resolvent}
\end{equation}
for every $f\in L^q(\mathbb R^3)$,
which follows from the fact that the range of $-\Delta$ is dense in $L^q(\mathbb R^3)$.
We fix $\varepsilon >0$ and set $w=\rho_\varepsilon *\psi$.
The resolvent parameter is denoted by $\lambda >0$, that will be chosen later as
$\lambda =1/k^2$.
Since $\mbox{div $\psi$}=0$, the resolvent $(\lambda -\Delta)^{-1}w$ is also solenoidal, so that
\[
W(\lambda):=
(-\Delta+\nabla\mbox{div})(\lambda -\Delta)^{-1}w
=-\Delta(\lambda -\Delta)^{-1}w.
\]
Then we have
\begin{equation*}
|x|\nabla (-\Delta+\nabla\mbox{div})\varphi_k(\lambda-\Delta)^{-1}w
=|x|\varphi_k\nabla W(\lambda)+R
\end{equation*}
with a remainder $R$ consisting of several terms, all of which involve derivatives of the cut-off function
$\varphi_k$.
We intend to prove
\begin{equation}
\lim_{k\to\infty}\big\||x|\varphi_k\nabla W(1/k^2)-|x|\nabla w\big\|_{q,\mathbb R^3}=0
\label{appro1w1}
\end{equation}
\begin{equation}
\lim_{k\to \infty}\|R\|_{q,\mathbb R^3}=0
\label{appro1w2}
\end{equation}
which imply \eqref{appro1w}.

On account of \eqref{resolvent}, we find that
\begin{equation*}
\begin{split}
&\quad \big\||x|\varphi_k\nabla W(\lambda)-|x|\nabla w\|_{q,\mathbb R^3} \\
&\leq\big\||x|\varphi_k\nabla (W(\lambda)-w)\big\|_{q,\mathbb R^3}
+\|(\varphi_k-1)|x|\nabla w\|_{q,\mathbb R^3}  \\
&\leq 2k\lambda\|\nabla (\lambda -\Delta)^{-1}w\|_{q,\mathbb R^3}
+\big\||x|\nabla w\big\|_{q,\{|x|> k\}}
\end{split}
\end{equation*}
with $\lambda =1/k^2$ goes to zero as $k\to \infty$.
This concludes \eqref{appro1w1}.
Moreover, we observe
\begin{equation*}
\begin{split}
&\quad \|R\|_{q,\mathbb R^3}  \\
&\leq C\big\||x||\nabla\varphi_k||\nabla^2(\lambda-\Delta)^{-1}w|\big\|_{q,A_k}
+C\big\||x||\nabla^2\varphi_k||\nabla(\lambda-\Delta)^{-1}w|\big\|_{q,\mathbb R^3}  \\
&\quad +C\big\||x||\nabla^3\varphi_k||(\lambda-\Delta)^{-1}w|\big\|_{q,\mathbb R^3}  \\
&\leq C\|\nabla^2(\lambda-\Delta)^{-1}w\|_{q,A_k}
+\frac{C}{k}\|\nabla(\lambda-\Delta)^{-1}w\|_{q,\mathbb R^3}
+\frac{C}{k^2}\|(\lambda-\Delta)^{-1}w\|_{q,\mathbb R^3}
\end{split}
\end{equation*}
where $A_k=\{x\in\mathbb R^3;\,k<|x|<2k\}$.
We then employ \eqref{resolvent} again to see that the second and third terms with 
$\lambda=1/k^2$ go to zero as $k\to\infty$.
Notice that the only balance is $\lambda=1/k^2$ to obtain the convergence of those terms 
as well as \eqref{appro1w1}.
In order to furnish \eqref{appro1w2}, it remains to show that so does the first term,
which needs slightly more argument (but the balance $\lambda=1/k^2$ is no longer necessary in what follows).
We set $(-\Delta)^{-1}w:=(4\pi|x|)^{-1}* w$, that solves the Poisson equation in $\mathbb R^3$
with the force $w$.
Then we have
\begin{equation*}
\|\nabla^2(\lambda-\Delta)^{-1}w\|_{q,A_k}
\leq \|\nabla^2[(\lambda-\Delta)^{-1}w-(-\Delta)^{-1}w]\|_{q,\mathbb R^3}
+\|\nabla^2(-\Delta)^{-1}w\|_{q,A_k}
\end{equation*}
where the latter term goes to zero as $k\to \infty$ since $w\in L^q(\mathbb R^3)$
implies $\nabla^2(-\Delta)^{-1}w\in L^q(\mathbb R^3)$, whereas the former term is
estimated from above by
\[
\|\Delta [(\lambda-\Delta)^{-1}w-(-\Delta)^{-1}w]\|_{q,\mathbb R^3}
=\lambda\|(\lambda-\Delta)^{-1}w\|_{q,\mathbb R^3}
\]
which goes to zero as $\lambda =1/k^2\to 0$ by \eqref{resolvent}.
The proof is complete.
\end{proof}

We are in a position to provide the integral equation \eqref{mild-wh} for $v(t)$ given by \eqref{modi0}.
\begin{proposition}
Let $u(t)$ be the solution obtained in Theorem \ref{main1} and
$v(t)$ the function given by \eqref{modi0}.
Then we have
\begin{equation}
v(t)=\int_{-\infty}^tU(t,s)P_{\mathbb R^3}
\big[g(s)-\mbox{\rm{div}$(v\otimes v)(s)$}\big]\,ds
\label{mild-wh}
\end{equation}
in $L^q_\sigma(\mathbb R^3)$
for all $t\in\mathbb R$ and $q\in (3,\infty)$, where $g$ is given by \eqref{g}
and $U(t,s)$ is the evolution operator for the whole space problem,
see \eqref{formula}--\eqref{funda-sol} for its representation.
\label{mild-whole}
\end{proposition}
\begin{proof}
First of all, along the same argument as in Lemma \ref{b-s}
with the aid of decay estimate of the composite operator $U(t,\tau)P_{\mathbb R^3}\,\mbox{div}$ as in \eqref{evo5},
where $q\geq 3/2$ is redundant,
the second term of the right-hand side of \eqref{mild-wh} makes sense as the Bochner integral in 
$L^q_\sigma(\mathbb R^3)$ 
for every $q\in (3,\infty)$
because of \eqref{v-cl2}.
The same thing for the other term involving $g$ follows from \eqref{v-g-est} with $r$ being
close to $1$ dependently on given $q\in (3,\infty)$ together with estimate of $U(t,s)$ as in \eqref{evo1}.

Let $1<r<3/2$.
By Lemma \ref{density} the class of test functions in \eqref{weak-wh} extends to
$\psi\in Z_r(\mathbb R^3)$.
In fact, given $\psi\in Z_r(\mathbb R^3)$, let us take $\psi_j\in C^\infty_{0,\sigma}(\mathbb R^3)$
such that $\|\psi_j-\psi\|_{Z_r(\mathbb R^3)}\to 0$ as $j\to\infty$.
Since $v(t),\,\nabla v(t),\,(v\otimes v)(t),\,g(t)\in L^{r^\prime}(\mathbb R^3)$ with $r^\prime >3$,
we have
\begin{equation*}
\begin{split}
\langle v(t),\psi_j\rangle
&\to \langle v(t),\psi\rangle,  \\
\frac{d}{dt}\langle v(t),\psi_j\rangle  
&\to
-\langle \nabla v+v\otimes (\eta+\omega\times x), \nabla\psi\rangle
-\langle \omega\times v, \psi\rangle
+\langle v\otimes v,\nabla\psi\rangle+\langle g,\psi\rangle
\end{split}
\end{equation*}
as $j\to\infty$, where the latter convergence is uniform with respect to $t\in\mathbb R$
in view of \eqref{v-cl2}--\eqref{v-cl3} and \eqref{v-g-est}.
We thus get $\langle v(\cdot),\psi\rangle\in C^1(\mathbb R)$ 
for all $\psi\in Z_r(\mathbb R^3),\, 1<r<3/2$,
together with \eqref{weak-wh}.

We fix $t\in\mathbb R$ and take
$\psi\in C^\infty_{0,\sigma}(\mathbb R^3)$, 
then we know that
$U(t,s)^*\psi\in Z_r(\mathbb R^3)$ 
for every $s\in (-\infty,t)$ and $r\in (1,\infty)$.
We can employ \eqref{weak-wh} with $\psi$ replaced by $U(t,s)^*\psi$ and \eqref{back-wh} to find
\begin{equation*}
\begin{split}
&\quad \frac{d}{ds}\langle U(t,s)v(s),\psi\rangle
=\frac{d}{ds}\langle v(s), U(t,s)^*\psi\rangle \\
&=-\langle \nabla v+v\otimes (\eta(s)+\omega(s)\times x), \nabla U(t,s)^*\psi\rangle
-\langle \omega(s)\times v, U(t,s)^*\psi\rangle \\
&\quad +\langle v\otimes v, \nabla U(t,s)^*\psi\rangle
+\langle g, U(t,s)^*\psi\rangle
+\langle v, L_{\mathbb R^3}(s)^*U(t,s)^*\psi\rangle \\
&=\langle (v\otimes v)(s), \nabla U(t,s)^*\psi\rangle
+\langle g(s), U(t,s)^*\psi\rangle  \\
&=\langle U(t,s)P_{\mathbb R^3}\big[g(s)-\mbox{div$(v\otimes v)(s)$}\big],\,\psi\rangle.
\end{split}
\end{equation*}
Integrating this from $\tau$ to $t$ and letting $\tau\to -\infty$ yield \eqref{mild-wh} in $L^q_\sigma(\mathbb R^3)$
for every $q\in (3,\infty)$ since
$\lim_{\tau\to -\infty}\|U(t,\tau)v(\tau)\|_{q,\mathbb R^3}=0$
follows from the same decay estimate of $U(t,\tau)$ as in \eqref{evo1} together with \eqref{v-cl2}
(as in the proof of Proposition \ref{q-bdd})
and since $\psi\in C_{0,\sigma}^\infty(\mathbb R^3)$ is arbitrary.
\end{proof}
\medskip

\noindent
{\bf 4.3.
Reconstruction procedure}
\medskip

Let $\{\Phi(t,s)\}_{t,s\in\mathbb R}$
be the family of evolution matrices to the ordinary differential equation
$\frac{d\phi}{dt}=-\omega(t)\times\phi$.
Since the right-hand side is skew-symmetric, each $\Phi(t,s)$ is an orthogonal $3\times 3$ matrix, which should be involved in the representation formula of 
the evolution operator $U(t,s)$ for the whole space problem. By using $U(t,s)$,
the solution to the inhomogeneous evolution equation 
\begin{equation}
\frac{du}{dt}+L_{\mathbb R^3}(t)u=P_{\mathbb R^3}g \qquad (t\in\mathbb R)
\label{evo-wh}
\end{equation}
with $g(t)$, possessing an appropriate behavior as $t\to -\infty$,
is described as
\begin{equation}
\begin{split}
u(x,t)&=\int_{-\infty}^t\big[U(t,s)P_{\mathbb R^3}g(s)\big](x)\,ds \\
&=\int_{-\infty}^t\int_{\mathbb R^3}K(x,y;\,t,s)g(y,s)\,dy\,ds
\end{split}
\label{formula}
\end{equation}
with
\begin{equation}
K(x,y;\,t,s):=\Phi(t,s)E\left(\Phi(t,s)^\top\Big(
x+\int_s^t\Phi(t,\tau)\eta(\tau)\,d\tau\Big)
-y,\;t-s\right)
\label{funda-sol}
\end{equation}
\begin{equation}
E(x,t):=(4\pi t)^{-3/2}e^{-|x|^2/4t}\,\mathbb I+\int_t^\infty(4\pi s)^{-3/2}\nabla^2
\left(e^{-|x|^2/4s}\right)\,ds,
\label{stokes-funda}
\end{equation}
where $\mathbb I$ denotes $3\times 3$ unity matrix. 
See \cite[Subsection 3.2]{Hi18}, \cite[Section 3]{Hi20} for this formula,
in which the Stokes fundamental solution $E(x,t)$ is replaced by the heat kernel (the first term of \eqref{stokes-funda})
since the initial value problem for the homogeneous equation is discussed in those papers.
It is 
seen that
\begin{equation}
|\nabla^jE(x,t)|\leq C_j(|x|^2+t)^{-(3+j)/2}
\label{est-stokes}
\end{equation}
for all $x\in\mathbb R^3$, $t>0$ and nonnegative integer $j$.

We begin with some auxiliary estimates, which are essentially known as estimates of
the Oseen fundamental solution $E(x+\zeta t,t)$ and some related potentials, see 
\cite{D1, Fa92, Fi65, G-b, GS07a, GS07b, KNP, Miz, Shi99}.
\begin{lemma}
Let $\zeta\in\mathbb R^3$. 
\begin{enumerate}

\item
We have
\begin{equation}
\begin{split}
\int_{\mathbb R^3}\int_0^\infty\frac{ds}{(|x+\zeta s-y|^2+s)^2}\;
&\frac{dy}{(1+|y|)^2(1+|\zeta||y|+\zeta\cdot y)^2}  \\
&\leq\frac{C}{(1+|x|)(1+|\zeta||x|+\zeta\cdot x)}
\end{split}
\label{auxi1}
\end{equation}
for all $x\in\mathbb R^3$ with some constant $C=C(|\zeta|)>0$.

\item
Let $3/2<q<\infty$ and $R>0$.
Assume that
$g\in L^\infty(\mathbb R;\,L^q(B_R))$.
Then
\begin{equation}
\int_0^\infty\int_{B_R}\frac{|g(y,t-s)|\,dy\,ds}{(|x+\zeta s-y|^2+s)^{3/2}}
\leq\frac{C\|g\|_{L^\infty(\mathbb R;\,L^q(B_R))}}{(1+|x|)(1+|\zeta||x|+\zeta\cdot x)}
\label{auxi2}
\end{equation}
for all $x\in\mathbb R^3$ and $t\in\mathbb R$ with some constant $C=C(|\zeta|,q,R)>0$.
\end{enumerate}
\label{lem-oseen}
\end{lemma}

\begin{proof}
By Galdi and Silvestre \cite[Lemma 1]{GS07a}
(see also Deuring \cite[Theorem 4.1]{D1}) 
it is known that
\begin{equation}
\begin{split}
&\int_0^\infty (|x+\zeta s|^2+s)^{-2}\,ds  \\
&\quad \leq\left\{
\begin{array}{ll}
C|\zeta|^{1/2}|x|^{-3/2}(1+|\zeta||x|+\zeta\cdot x)^{-3/2}, \quad 
& |x|>(4|\zeta|)^{-1}, \\
C|x|^{-2}, & 0<|x|\leq (4|\zeta|)^{-1},
\end{array}
\right.
\end{split}
\label{oseen-grad}
\end{equation}
where the right-hand side is understood as $C|x|^{-2}$ for all $x\in \mathbb R^3\setminus\{0\}$ if $\zeta=0$.
With \eqref{oseen-grad} at hand,
\eqref{auxi1} follows from results due to 
Farwig \cite[Lemma 3.1]{Fa92}, 
Kra\v cmar, Novotn\'y and Pokorn\'y \cite[Theorem 3.2]{KNP},
Galdi \cite[Lemma VIII.3.5]{G-b}.

When $\zeta=0$, estimate \eqref{auxi2} for $|x|> 2R$ is easy.
Let $\zeta\in\mathbb R^3\setminus\{0\}$.
Deuring \cite[Theorem 4.1, Corollary 4.1]{D1}
showed \eqref{auxi2} with $\int_0^\infty$ 
replaced by $\int_0^t$
for $|x|>2R$.
If we take into account 
(4.6)--(4.7) of \cite{D1}, then we see that his estimate for $|x|>2R$
is also valid for our integral. 
But it would be better to give a brief sketch of his proof for convenience of readers.
In addition, we need \eqref{auxi2} in the whole $\mathbb R^3$; in fact, the condition
$q>3/2$ is necessary for boundedness locally in $\mathbb R^3$.

Let $|x|>2R$.
Then the key observation due to Deuring \cite[(4.5)--(4.7)]{D1} is
\begin{equation}
\begin{split}
(|x+\zeta s-y|^2+s)^{1/2}
&\geq c\left[\Big||x-y|-|\zeta|s\Big|+\big\{s\,(1+|\zeta||x|+\zeta\cdot x)\big\}^{1/2}\right]  \\
&\geq c\,F(x,s)>0
\end{split}
\label{deu-ineq}
\end{equation}
for $|x|>2R$, $y\in B_R$ and $s>0$, where
$c=\min\big\{\frac{1}{\sqrt 2},\frac{1}{\sqrt{1+2|\zeta|R}}\big\}$ and
\begin{equation*}
F(x,s):=
\left\{
\begin{array}{ll}
|\zeta|s-(|x|+R)+\big[\frac{|x|}{|\zeta|}(1+|\zeta||x|+\zeta\cdot x)\big]^{1/2} \quad &(s>\frac{|x|+R}{|\zeta|}), \\
\big[\frac{|x|}{2|\zeta|}(1+|\zeta||x|+\zeta\cdot x)\big]^{1/2} & (\frac{|x|-R}{|\zeta|}<s\leq\frac{|x|+R}{|\zeta|}), \\
|x|-R-|\zeta|s+\big[\frac{|x|}{4|\zeta|}(1+|\zeta||x|+\zeta\cdot x)\big]^{1/2} &(\frac{|x|}{4|\zeta|}<s\leq\frac{|x|-R}{|\zeta|}), \\
\frac{|x|}{2}-|\zeta|s &(0<s\leq \frac{|x|}{4|\zeta|}),
\end{array}
\right.
\end{equation*}
which enjoys
\begin{equation}
\int_0^\infty F(x,s)^{-3}\,ds\leq C|x|^{-1}(1+|\zeta||x|+\zeta\cdot x)^{-1}
\label{deu2}
\end{equation}
for $|x|>2R$ with some positive constant $C=C(|\zeta|,R)>0$.
After use of the H\"older inequality in $y$ with the assumption on $g$, we apply
\eqref{deu-ineq}--\eqref{deu2} to obtain \eqref{auxi2} for $|x|>2R$
when $\zeta\in\mathbb R^3\setminus\{0\}$.

We next discuss the case $|x|\leq 2R$, for which it suffices to show the boundedness of the
left-hand side of \eqref{auxi2}. 
Indeed, it is uniformly bounded in the whole $\mathbb R^3$
with respect to $\zeta\in\mathbb R^3$.
We first observe that if $\alpha>3/2$, then
\begin{equation}
\sup_{x\in\mathbb R^3}\int_{\mathbb R^3}(|x+\zeta s-y|^2+s)^{-\alpha}\,dy
=\int_{\mathbb R^3}(|z|^2+s)^{-\alpha}\,dz
=C_\alpha\, s^{-\alpha+3/2}
\label{int-y0}
\end{equation}
for all $s>0$ 
by the change of variable $z=y-x-\zeta s$. 
We divide the left-hand side of \eqref{auxi2} into
$\left(\int_0^1+\int_1^\infty\right)\int_{B_R}$. 
We then employ the H\"older inequality in $y$ to take $\|g(t-s)\|_{q,B_R}$
with $q\in (3/2,\infty)$ for the former and the one to take $\|g(t-s)\|_{r,B_R}$ with fixed $r\in (1,3/2)$ for the latter.
Using \eqref{int-y0}
with $\alpha=3q^\prime/2$ and $\alpha=3r^\prime/2$, respectively,
we find that 
the left-hand side of \eqref{auxi2} is bounded from above by
\[
C\|g\|_{L^\infty(\mathbb R;\, L^q(B_R))}
\int_0^1s^{-3/2q}\,ds +
C\|g\|_{L^\infty(\mathbb R;\, L^r(B_R))}
\int_1^\infty s^{-3/2r}\,ds,
\]
which completes the proof.
\end{proof}

Let $m>0$ and $\zeta\in\mathbb R^3$.
Following Galdi \cite{G-new}, let us introduce the anisotropic weighted space
\begin{equation*}
X_{m,\zeta}:=\big\{f\in 
L^\infty(\mathbb R^3_x\times \mathbb R_t);\;
(1+|x|)^m(1+|\zeta||x|+\zeta\cdot x)^m f\in L^\infty(\mathbb R^3_x\times \mathbb R_t)
\big\}
\end{equation*}
which is a Banach space endowed with norm
\[
[f]_{m,\zeta}:=
\left\|(1+|x|)^m(1+|\zeta||x|+\zeta\cdot x)^m f\right\|_{L^\infty(\mathbb R^3\times \mathbb R)}.
\]
\begin{lemma}
Set
\begin{equation}
\begin{split}
&(Sg)(\cdot,t)=\int_{-\infty}^tU(t,s)P_{\mathbb R^3}g(s)\,ds,  \\
&(\Lambda G)(\cdot,t)=-\int_{-\infty}^tU(t,s)P_{\mathbb R^3}\mbox{\rm{div} $G(s)$}\,ds.
\end{split}
\label{op-wh}
\end{equation}
Assume \eqref{wake-cond} for some $\zeta\in\mathbb R^3$.

\begin{enumerate}
\item
Let $3/2<q<\infty$ and
$g\in L^\infty(\mathbb R;\,L^q(\mathbb R^3))$ with 
$g=0$ \rm{a.e.}$(\mathbb R^3\setminus B_2)\times\mathbb R$.
Then $Sg\in X_{1,\zeta}$ subject to
\begin{equation}
[Sg]_{1,\zeta}
\leq C_0
\|g\|_{L^\infty(\mathbb R;\,L^q(\mathbb R^3))}
\label{est-wake1}
\end{equation}
with some constant $C_0=C_0(|\zeta|,M,q)>0$ independent of $g$,
where $M$ is given by \eqref{wake-cond}.

\item
Assume further that $\omega(t)$ is parallel to $\zeta$ for every $t\in\mathbb R$ if $\zeta\neq 0$.
Let $G\in X_{2,\zeta}$.
Then $\Lambda G\in X_{1,\zeta}$ subject to
\begin{equation}
[\Lambda G]_{1,\zeta}\leq C_1[G]_{2,\zeta}
\label{est-wake2}
\end{equation}
with some constant $C_1=C_1(|\zeta|,M)>0$ independent of $G$,
where $M$ is given by \eqref{wake-cond}.
\end{enumerate}
\label{est-funda}
\end{lemma}

\begin{proof}
In view of \eqref{formula}--\eqref{funda-sol} we make the change of variable
\begin{equation}
y\mapsto z=\Phi(t,s)y-\int_s^t\big\{\Phi(t,\tau)\eta(\tau)-\zeta\big\}\,d\tau
\label{change}
\end{equation}
to obtain
\[
(Sg)(x,t)=
\int_{-\infty}^t\int_{\mathbb R^3}
\Phi(t,s)E\left(\Phi(t,s)^\top\Big(x+\zeta (t-s)-z\Big),\;t-s\right)\widetilde g(z,s)\,dz\,ds
\]
with
$\widetilde g(z,t):=g(y,t)$.
By 
the assumption \eqref{wake-cond} the function $\widetilde g(t)$ still 
satisfies
\begin{equation*}
\begin{split}
&\widetilde g\in L^\infty(\mathbb R;\,L^q(\mathbb R^3)), \qquad
\widetilde g=0 \quad\mbox{a.e.}\; 
(\mathbb R^3\setminus B_{2+M})\times\mathbb R,  \\
&\sup_{t\in \mathbb R}\|\widetilde g(t)\|_{q,\mathbb R^3}
=\sup_{t\in\mathbb R}\|g(t)\|_{q,\mathbb R^3}.
\end{split}
\end{equation*}
Then \eqref{est-wake1} follows from \eqref{est-stokes} and \eqref{auxi2}.

Let us show the second assertion.
By the change of variable \eqref{change} after integration by parts in \eqref{formula} with $g=\mbox{div $G$}$, the function
$(\Lambda G)(x,t)$ is rewritten as
\begin{equation*}
\begin{split}
&\quad (\Lambda G)(x,t)  \\
&=\int_{-\infty}^t\int_{\mathbb R^3}\nabla_yK(x,y;\,t,s):G(y,s)\,dy\,ds \\
&=-\int_{-\infty}^t\int_{\mathbb R^3}\Phi(t,s)(\nabla E)
\left(\Phi(t,s)^\top\Big(x+\zeta(t-s)-z\Big),\; t-s\right)
\widetilde G(z,s)\,dz\,ds
\end{split}
\end{equation*}
with
$\widetilde G(z,t):=G(y,t)$, that satisfies
\begin{equation}
\begin{split}
|\widetilde G(z,t)|
&\leq [G]_{2,\zeta}\,(1+|y|)^{-2}(1+|\zeta||y|+\zeta\cdot y)^{-2}  \\
&\leq c_*[G]_{2,\zeta}\,(1+|z|)^{-2}(1+|\zeta||z|+\zeta\cdot z)^{-2}\quad
\mbox{a.e. $\mathbb R^3\times \mathbb R$}.
\end{split}
\label{wake-inv}
\end{equation}
In fact, 
since $\omega(t)$ is parallel to $\zeta$, we have
$\Phi(t,s)^\top\zeta=\zeta$, which implies
$\zeta\cdot(\Phi(t,s)y-y)=0$, so that
\[
|\zeta\cdot (z-y)|=|\zeta\cdot (z-\Phi(t,s)y)|\leq M|\zeta|
\]
by virtue of 
\eqref{wake-cond}.
We thus observe
\begin{equation*}
\begin{split}
(1+|z|)(1+|\zeta||z|+\zeta\cdot z)
&\leq (1+|y|+M)(1+|\zeta||y|+\zeta\cdot y+2M|\zeta|) \\
&\leq (1+M)(1+2M|\zeta|)(1+|y|)(1+|\zeta||y|+\zeta\cdot y) 
\end{split}
\end{equation*}
which leads to \eqref{wake-inv} with
$c_*=(1+M)^2(1+2M|\zeta|)^2$.
Then \eqref{est-wake2} follows from \eqref{est-stokes} and \eqref{auxi1}.
The proof is complete.
\end{proof}
\begin{remark}
From the proof we see that
\begin{equation*}
\begin{split}
&(1+|x|)(1+|\zeta||x|+\zeta\cdot x)|(Sg)(x,t)|
\leq C_0\|g\|_{L^\infty(\mathbb R;\, L^q(\mathbb R^3))} \\
&(1+|x|)(1+|\zeta||x|+\zeta\cdot x)|(\Lambda G)(x,t)|
\leq C_1[G]_{2,\zeta}
\end{split}
\end{equation*}
hold true for all $(x,t)\in\mathbb R^3\times \mathbb R$ rather than \eqref{est-wake1}--\eqref{est-wake2}.
\label{rem-pointwise}
\end{remark}
\noindent
{\it Proof of Theorem \ref{main3}}.
By using the operators $S$ and $\Lambda$ given by \eqref{op-wh},
the integral equation \eqref{mild-wh} is written as
\begin{equation}
v=Sg+\Lambda(v\otimes v)
\label{eq-wh}
\end{equation}
where the function $g(t)$ is given by \eqref{g} 
and fulfills \eqref{v-g-est}--\eqref{G-est}. 
Let us fix $r\in (3/2,\infty)$.
On account of \eqref{est-wake1}--\eqref{est-wake2} and by means of a contraction mapping principle,
it is easy to construct a solution
$\widetilde v\in X_{1,\zeta}$ with
\[
[\widetilde v]_{1,\zeta}
\leq \frac{1-\sqrt{1-4C_0C_1k_g}}{2C_1}
<2C_0k_g, \qquad
k_g:=\|g\|_{L^\infty(\mathbb R;\, L^r(\mathbb R^3))}
\]
to the equation \eqref{eq-wh} provided $k_g<1/4C_0C_1$.
By virtue of \eqref{v-g-est}
this smallness condition is indeed accomplished if $\|(\eta,\omega)\|_{W^{1,\infty}}$ is still smaller.
Note that $\widetilde v\in L^\infty(\mathbb R;\, L^{3,\infty}_\sigma(\mathbb R^3))$ and that
\begin{equation}
\sup_{t\in\mathbb R}|\widetilde v(x,t)|\leq C\|(\eta,\omega)\|_{W^{1,\infty}}
(1+|x|)^{-1}(1+|\zeta||x|+\zeta\cdot x)^{-1} \;\;
\mbox{for all $x\in\mathbb R^3$}
\label{v-tilde}
\end{equation}
by taking also into account Remark \ref{rem-pointwise}.

As in the proof of uniqueness of solutions in Theorem \ref{main1} by use of \eqref{G-est}
and thanks to the same estimate for $\nabla U(t,s)^*$ as in \eqref{evo4},
the solution to \eqref{mild-wh} 
(even to the corresponding
weak form as in \eqref{weak-int})
is unique within the class
$L^\infty(\mathbb R;\,L^{3,\infty}_\sigma(\mathbb R^3))$ with small norm,
see Remark \ref{sol-bdd} (i).
Hence, as long as $\|(\eta,\omega)\|_{W^{1,\infty}}$ is small enough,
the function $v(t)$, which is given by \eqref{modi0} and fulfills \eqref{v-cl1},
coincides with $\widetilde v(t)$ reconstructed above. 
In view of \eqref{v-tilde}, we are led to
the desired pointwise decay \eqref{pointwise} at infinity.
The proof is complete.
\hfill
$\Box$

\end{document}